\documentclass[oneside,leqno,11pt]{article}
\usepackage{amssymb,amsmath,latexsym,amsthm}
\newtheorem{theorem}[equation]{Theorem}

\newtheorem{lemma}[equation]{Lemma}

\newtheorem{proposition}[equation]{Proposition}

\def\sideremark#1{\ifvmode\leavevmode\fi\vadjust{\vbox to0pt{\vss
 \hbox to 0pt{\hskip\hsize\hskip1em
\vbox{\hsize2cm\tiny\raggedright\pretolerance10000 
 \noindent #1\hfill}\hss}\vbox to8pt{\vfil}\vss}}} 

\def\Ad{{\rm Ad}\,}
\def\ad{{\rm ad}\,}

\begin{document}


\title{Killing Vector Fields of Constant Length on Riemannian Normal
        Homogeneous Spaces}

\author{Ming Xu\footnote{Address: College of Mathematics,
        Tianjin Normal University,
        Tianjin 300387, P.R.China; e--mail: {\tt mgmgmgxu@163.com}.
        Research supported by NSFC no. 11271216, State Scholarship
        Fund of CSC (no. 201408120020), Science and Technology Development
Fund for Universities and Colleges in Tianjin
(no. 20141005), Doctor fund of Tianjin Normal
        University (no. 52XB1305).}\,\, \&
        Joseph A. Wolf\footnote{ Corresponding author.
        Address: Department of Mathematics, University of California, Berkeley,
        CA 94720--3840; e--mail: {\tt jawolf@math.berkeley.edu}.
        Research partially supported by a Simons Foundation grant and by the
        Dickson Emeriti Professorship at the University of California,
        Berkeley.}}

\date{}

\maketitle

\begin{abstract}
Killing vector fields of constant length correspond to isometries of
constant displacement.  Those in turn have been used to study homogeneity
of Riemannian and Finsler quotient manifolds.  Almost all of that work 
has been done for group manifolds or, more generally, for symmetric
spaces.  This paper extends the scope of research on constant length
Killing vector fields to a class of Riemannian normal homogeneous spaces.
\end{abstract}

\section{Introduction}
\setcounter{equation}{0}
An isometry $\rho$ of a metric space $(M,d)$ is called
Clifford--Wolf (CW) if
it moves each point the same distance, i.e. if the displacement function
$\delta(x) = d(x,\rho(x))$ is constant.  W. K. Clifford \cite{C1873}
described such isometries for the $3$--sphere, and G. Vincent
\cite{V1947} used the term {\em Clifford translation} for constant
displacement isometries of round spheres in his study of spherical
space forms $S^n/\Gamma$ with $\Gamma$ metabelian.
Later J. A. Wolf (\cite{W1960}, \cite{W1961a}, \cite{W1961b}) extended the
use of the term {\em Clifford translation} to the context of metric spaces,
especially Riemannian symmetric spaces.  There the point is his theorem
\cite{W1962b} that a complete locally
symmetric Riemannian manifold $M$ is homogeneous if and only if, in the
universal cover $\widetilde{M} \to M = \Gamma \backslash \widetilde{M}$,
the covering group $\Gamma$ consists of Clifford translations.  In part Wolf's
argument was case by case, but later V. Ozols (\cite{O1969},\cite{O1973},
\cite{O1974b}) gave a general argument for the
situation where $\Gamma$ is a cyclic subgroup of the identity component
$I^0(\widetilde{M})$ of the isometry group  $I(\widetilde{M})$.
H. Freudenthal \cite{F1963} discussed the situation where
$\Gamma \subset I^0(\widetilde{M})$, and introduced the term
{\em Clifford--Wolf isometry} (CW) for isometries of constant displacement.
That seems to be the term in general usage.
More recently, the result \cite{W1962b} for locally symmetric homogeneous
Riemannian manifolds was extended to Finsler manifolds by S. Deng and
J. A. Wolf \cite{DW2012}.
\smallskip

In the setting of non--positive sectional curvature, isometries of bounded
displacement are already CW \cite{W1964}.
Further there has been some work relating CW and homogeneity
for pseudo--Riemannian manifolds (\cite{W1961b},\cite{W1962a}).
\smallskip

Recently, V.~N.~Berestovskii and Yu.~G.~Nikonorov classified
all simply connected Riemannian homogeneous spaces such that the homogeneity
can be achieved by CW translations, i.e. CW homogeneous spaces (\cite{BN2008a},
\cite{BN2008b}, \cite{BN2009}). Also, S.~Deng and M.~Xu studied CW
isometries and CW homogeneous spaces in Finsler geometry (\cite{DX2012},
\cite{DX2013a}, \cite{DX2013b}, \cite{DX2013c}, \cite{DX2014a}, \cite{DX2014b}).
\smallskip

Most of the research on CW translations has been concerned with Riemannian
(and later Finsler) symmetric spaces.  There we have a full understanding
of CW translations (\cite{W1962b} and \cite{DW2012}), but little is known about
CW translations on non-symmetric homogeneous Riemannian spaces. For example,
there are not many examples of Clifford-Wolf translations on Riemannian
normal homogeneous space $G/H$ with $G$ compact simple, except the those
found on Riemannian symmetric spaces and in some closely related settings
(see \cite{DMW1986} and \cite{D1983}).
\smallskip

The infinitesimal version, apparently introduced by
V. N. Berestovskii and Yu. G. Nikonorov, is that of Killing vector fields
of constant length.  We will refer to those Killing vector fields as
{\em Clifford--Killing} vector fields or CK vector fields.  They correspond
(at least locally) to one parameter local groups of CW isometries.  The purpose
of this work is to study, and classify all CK vector fields on Riemannian
normal homogeneous spaces $M=G/H$.
\smallskip

We recall the general definition of Riemannian normal homogeneous spaces.
Let $G$ be a connected Lie group and $H$ a compact subgroup,
such that $M=G/H$ carries
a $G$--invariant Riemannian metric.  Thus the Lie algebra $\mathfrak{g}$
has an $\Ad(H)$--invariant direct sum decomposition
$\mathfrak{g}=\mathfrak{h} + \mathfrak{m}$ where the natural projection $\pi: G \to G/H$
maps $\mathfrak{m}$ onto the tangent space at the base point $o = \pi(\mathbf{e})$,
and the Riemannian metric corresponds to a positive definite inner product
$\langle \cdot \,, \cdot \rangle_\mathfrak{m}$ on $\mathfrak{m}$.
The Riemannian manifold $M$ is called {\em naturally reductive} if the
$\Ad(H)$--invariant decomposition $\mathfrak{g} = \mathfrak{h} + \mathfrak{m}$
can be chosen so that
$$
\langle \mathfrak{pr}_{\mathfrak{m}}[u,v],w\rangle_\mathfrak{m}
 + \langle v, \mathfrak{pr}_{\mathfrak{m}}[u,w]\rangle_\mathfrak{m} = 0
\text{ for all } u,v,w \in \mathfrak{m}
$$
where $\mathfrak{pr}_{\mathfrak{m}}: \mathfrak{g} \to \mathfrak{m}$
is the projection with kernel $\mathfrak{h}$.  When $\langle \cdot \,, \cdot \rangle_\mathfrak{m}$ is
the restriction of an $\Ad(G)$--invariant nondegenerate symmetric bilinear form $\langle \cdot \,, \cdot \rangle$ on $\mathfrak{g}$ such that
$\mathfrak{h} \perp \mathfrak{m}$,\,
$M = G/H$ is a {\em Riemannian normal homogeneous space}.
In this general definition, Riemannian normal homogeneous space is viewed as a
generalization of Riemannian symmetric space, including the non-compact type.
If we expect the Riemannian normal homogeneous space to be related to Riemannian isometric
submersions, there is another definition of Riemannian normal homogeneous space, which
require $\langle\cdot,\cdot\rangle_\mathfrak{m}$ is the restriction of
an $\Ad(G)$-invariant (bi-invariant) inner product on $\mathfrak{g}$. In this definition,
$\mathfrak{g}$ must be compact, i.e. $G$ is quasi-compact, or equivalently
the universal cover of $G$ is the product of a compact semi-simple Lie group and an Euclidean
space (which can be 0).
Note that, in both definition, the normal homogeneous metric on $M$ depends on $G$.
\smallskip

In this work, we will only consider the special case that $G$ is a \
compact connected simple Lie group.  Our main result is

\begin{theorem} \label{main}
Let $G$ be a compact connected simple Lie group and $H$ a closed subgroup with
$0<\dim H<\dim G$.  Fix a normal Riemannian metric on $M = G/H$ . Suppose
that there is a nonzero vector $v\in\mathfrak{g}$
defining a CK vector field on $M = G/H$. Then $M$ is a complete
locally symmetric Riemannian manifold, and its universal Riemannian cover is
an odd dimensional sphere of constant curvature or a Riemannian symmetric
space $SU(2n)/Sp(n)$.
\end{theorem}
\smallskip

It is obvious to see that when $\dim H=\dim G$, $M$ is reduced to a single
point, and when $\dim H=0$,
$M$ is locally Riemannian symmetric because it is covered by $G$ with the
bi-invariant Riemannian metric.

Riemannian normal homogeneity is a much weaker condition than Riemannian
symmetry
or locally Riemannian symmetric homogeneity. Even in the case where $G$ is a
compact connected simple 
Lie group, every smooth coset space $G/H$ has at least one invariant
normal Riemannian metric, while of course the list of Riemannian symmetric
spaces $G/H$ is rather short.  But Theorem \ref{main} provides the same 
classification result for CK vector fields.  It suggests that the existence 
of nontrivial CK vector fields and CW translations
will impose very stronger restrictions on a Riemannian homogeneous space,
at least when that space is Riemannian normal homogeneous.
\smallskip

On the other hand, we do not have have comprehensive results when $G$ is of 
non-compact type. When $G$ is compact but not simple,
generally speaking, a Riemannian normal homogeneous space $M$
does not have a perfect local decomposition into symmetric spaces.
Thus the study of CW translations and
CK vector fields in this situation is still open.
\smallskip

The proof of the Theorem \ref{main} is organized as follows. In Section 2, 
we summarize the notations and preliminaries for the Riemannian normal 
homogeneous spaces we will consider. In Section 3, we present some
preliminary lemmas, study the CK vector fields at the Cartan subalgebra 
level, and prove Theorem \ref{main} in the easiest situations. In 
Section 4, we prove the Theorem \ref{main} when $G$ is an exceptional 
Lie group. From Section 5 to Section 8, we prove the Theorem
\ref{main} when $G$ is a classical Lie group, i.e. of type 
$\mathfrak{a}_n$, $\mathfrak{b}_n$,
$\mathfrak{c}_n$ or $\mathfrak{d}_n$\,.
\smallskip

The first author thanks the Department of Mathematics at the University
of California, Berkeley, for hospitality during the preparation of this
paper.

\section{Notations about normal homogeneous spaces}
\setcounter{equation}{0}
Let $G$ be a compact connected simple Lie group and $H$ a closed subgroup
with $0<\dim H<\dim G$. We denote $\mathfrak{g}=\mathrm{Lie}(G)$ and
$\mathfrak{h}=\mathrm{Lie}(H)$. Fix a bi-invariant inner product
$\langle\cdot,\cdot\rangle$ of $\mathfrak{g}$. It defines
a decomposition $\mathfrak{g}=\mathfrak{h}+\mathfrak{m}$ with
$\mathfrak{h}\perp\mathfrak{m}$ and
$[\mathfrak{h},\mathfrak{m}]\subset\mathfrak{m}$.
The orthogonal projection to each factor is denoted
$\mathrm{pr}_{\mathfrak{h}}$ or $\mathfrak{pr}_{\mathfrak{m}}$ respectively.
We naturally identify $\mathfrak{m}$ with the tangent space $T_o(G/H)$ at
$o=\pi(\mathbf{e})$ where $\pi : G \to G/H$ is the usual projection.
The restriction of $\langle\cdot,\cdot\rangle$ to $\mathfrak{m}$
is $\mathrm{Ad}(H)$-invariant and defines a $G$-invariant Riemannian metric
on $M=G/H$. A Riemannian metric defined in this way is called a
{\em normal homogeneous metric}, and $M$ together with a normal
homogeneous metric is a {\em Riemannian normal homogeneous space}.  Note
the dependence on $G$.  Here $G$ is simple so
the normal homogeneous metric on $G/H$ is unique up to scalar multiplications.
\smallskip

Any Cartan subalgebra of $\mathfrak{h}$ can be expanded to a Cartan subalgebra
$\mathfrak{t}$ of $\mathfrak{g}$ such that
$\mathfrak{t}=\mathfrak{t}\cap\mathfrak{h}+\mathfrak{t}\cap\mathfrak{m}$.
As any two Cartan subalgebras of $\mathfrak{g}$ are conjugate
we can assume $\mathfrak{t}$ is the standard one.
For example when $\mathfrak{g}=\mathfrak{su}(n+1)$, $\mathfrak{t}$ is
the subalgebra of all diagonal matrices.
The standard $\mathfrak{u}(n)\hookrightarrow \mathfrak{so}(2n)$
comes from $a+b\sqrt{-1}\mapsto \left( \begin{smallmatrix}
a & b \\ -b & a \\ \end{smallmatrix} \right)$
for $a,b \in \mathbb{R}$.  We view $\mathfrak{sp}(n)$
as the space of all skew--Hermitian skew $n\times n$ quaternion matrices
where $q \mapsto \overline{q}$ is the usual conjugation of $\mathbb{H}
= \mathbb{R} + \mathbb{R}\mathbf{i} + \mathbb{R}\mathbf{j} +
\mathbb{R}\mathbf{k}$ over $\mathbb{R}$.  Then
$\mathfrak{u}(n)\subset\mathfrak{sp}(n)$ when we identify
$\sqrt{-1}$ with $\mathbf{i}$.  With these descriptions
the space of diagonal matrices in $\mathfrak{u}(n)$ also provides the
standard Cartan subalgebra $\mathfrak{t}$ for the other classical compact
simple Lie algebras. The standard Cartan subalgebra of $\mathfrak{so}(2n)$
can also be viewed as that for $\mathfrak{so}(2n+1)$ with $\mathfrak{so}(2n)$
identified with the block at the right down corner.
\smallskip

Let $\Delta=\Delta(\mathfrak{g},\mathfrak{t})$ be the root system of
$\mathfrak{g}$, and $\Delta^+$
be any positive root system in $\Delta$. Because of the bi-invariant inner
product on $\mathfrak{g}$, the roots of $\mathfrak{g}$ can be viewed as
elements of $\mathfrak{t}$ instead of $\mathfrak{t}^*$.
We have the standard decomposition of $\mathfrak{g}$,
\begin{equation}\label{decomp-1}
\mathfrak{g}=\mathfrak{t}+{\sum}_{\alpha\in\Delta^+}\mathfrak{g}_{\pm\alpha},
\end{equation}
in which each $\mathfrak{g}_{\pm\alpha}$ is the real two dimensional root plane
$(\mathfrak{g}^C_\alpha + \mathfrak{g}^C_{-\alpha})\cap \mathfrak{g}$.
Considering the subalgebra $\mathfrak{h}$, we have another decomposition of
$\mathfrak{g}$:
\begin{equation}\label{decomp-2}
\mathfrak{g}=
	\mathfrak{t}+{\sum}_{\alpha'\in\mathrm{pr}_{\mathfrak{h}}(\Delta^+)}\,
	\widehat{\mathfrak{g}}_{\pm\alpha'}
\text{ where } \widehat{\mathfrak{g}}_{\pm\alpha'}=
	{\sum}_{\alpha\in\Delta^+,\mathrm{pr}_{\mathfrak{h}}(\pm\alpha)=
	\pm\alpha'}\,\,\mathfrak{g}_{\pm\alpha}\,.
\end{equation}
Both (\ref{decomp-1}) and (\ref{decomp-2}) are orthogonal decompositions.
More importantly, we have orthogonal decompositions
\begin{eqnarray}
\mathfrak{t}&=&(\mathfrak{t}\cap\mathfrak{h})+(\mathfrak{t}\cap\mathfrak{m})
\mbox{ and}\label{decomp-3}\\
\widehat{\mathfrak{g}}_{\pm\alpha'}&=&(\widehat{\mathfrak{g}}_{\pm\alpha'}\cap\mathfrak{h})+
(\widehat{\mathfrak{g}}_{\pm\alpha'}\cap\mathfrak{m}),\label{decomp-4}
\end{eqnarray}
in which the summands are equal to the images of the projection maps
$\mathrm{pr}_{\mathfrak{h}}$ and $\mathrm{pr}_{\mathfrak{m}}$.
\smallskip

Let $\Delta'=\Delta(\mathfrak{h},\mathfrak{t}\cap\mathfrak{h})$
denote the root system of $\mathfrak{h}$ and choose a positive subsystem
$\Delta'^+\subset\Delta'$. The restriction of the bi-invariant inner product
of $\mathfrak{g}$ to $\mathfrak{h}$ is a bi-invariant inner product there,
and $\Delta'$ can be viewed as a subset
of $\mathfrak{t}\cap\mathfrak{h}$. For each root $\alpha'\in\Delta'^+$,
the two dimensional root plane $\widehat{\mathfrak{h}}_{\pm\alpha'}$ is just
the factor $\widehat{\mathfrak{g}}_{\pm\alpha'}\cap\mathfrak{h}$ in
(\ref{decomp-4}).
\smallskip

For each simple Lie algebra $\mathfrak{g}$ we recall the Bourbaki
description of the root system $\Delta^+$, and the root planes in the
classical cases.
\smallskip

(1) The case $\mathfrak{g}=\mathfrak{a}_n=\mathfrak{su}(n+1)$ for $n>0$.
Let $\{e_1,\ldots,e_{n+1}\}$ denote the standard orthonormal basis of
$\mathbb{R}^{n+1}$\,.  Then $\mathfrak{t}$ can be isometrically identified with
the subspace $(e_1+\cdots+e_{n+1})^\perp \subset \mathbb{R}^{n+1}$. The
root system $\Delta$ is
\begin{equation}\label{root-system-A-n}
\{\pm(e_i-e_j) \mid 1\leqq i < j\leqq n+1\}.
\end{equation}
Let $E_{i,j}$ be the matrix with all zeros except for
a $1$ in the $(i,j)$ place.  Then
\begin{eqnarray*}
e_i&=&\sqrt{-1}E_{i,i}\in\mathfrak{su}(n+1), \mbox{ and} \\
\mathfrak{g}_{\pm(e_i-e_j)}&=&\mathbb{R}(E_{i,j}-E_{j,i})+\mathbb{R}\sqrt{-1}(E_{i,j}+E_{j,i}).
\end{eqnarray*}
\smallskip

(2) The case $\mathfrak{g}=\mathfrak{b}_n=\mathfrak{so}(2n+1)$ for $n>1$.
The Cartan subalgebra $\mathfrak{t}$ can be isometrically identified with
$\mathbb{R}^n$ with the standard orthonormal basis $\{e_1,\ldots,e_n\}$.
The root system $\Delta$ is
\begin{equation}\label{root-system-B-n}
\{\pm e_i \mid 1\leqq i\leqq n\} \cup
	\{\pm e_i\pm e_j \mid 1\leqq i<j\leqq n\}.
\end{equation}
Using matrices, we have
\begin{eqnarray*}
e_i&=&E_{2i,2i+1}-E_{2i+1,2i}, \\
\mathfrak{g}_{\pm e_i}&=&\mathbb{R}(E_{2i,1}-E_{1,2i})+\mathbb{R}(E_{2i+1,1}-E_{1,2i+1}),\\
\mathfrak{g}_{\pm(e_i-e_j)}&=&\mathbb{R}(E_{2i,2j}+E_{2i+1,2j+1}
-E_{2j,2i}-E_{2j+1,2i+1})\\
&&\phantom{X}+\mathbb{R}(E_{2i,2j+1}-E_{2i+1,2j}+E_{2j,2i+1}-E_{2j+1,2i}),
\mbox{ and}\\
\mathfrak{g}_{\pm(e_i+e_j)}&=&
\mathbb{R}(E_{2i,2j}-E_{2i+1,2j+1}-E_{2j,2i}+E_{2j+1,2i+1})\\
&&\phantom{X}+\mathbb{R}(E_{2i,2j+1}+E_{2i+1,2j}-E_{2j,2i+1}-E_{2j+1,2i}).
\end{eqnarray*}
\smallskip

(3) The case $\mathfrak{g}=\mathfrak{c}_n=\mathfrak{sp}(n)$ for $n>2$.
As before $\mathfrak{t}$ is isometrically identified with
$\mathbb{R}^n$ with the standard orthonormal basis $\{e_1,\ldots,e_n\}$.
The root system $\Delta$ is
\begin{equation}\label{root-system-C-n}
\{\pm 2e_i \mid 1\leqq i\leqq n\} \cup
	\{\pm e_i\pm e_j \mid 1\leqq i<j\leqq n\}.
\end{equation}
Using matrices, we have
\begin{eqnarray*}
e_i &=&\mathbf{i}E_{i,i},\\
\mathfrak{g}_{\pm 2e_i}&=&\mathbb{R}\mathbf{j}E_{i,i}+\mathbb{R}\mathbf{k}E_{i,i},\\
\mathfrak{g}_{\pm(e_i-e_j)}&=&\mathbb{R}(E_{i,j}-E_{j,i})+\mathbb{R}\mathbf{i}(E_{i,j}+E_{j,i}),\mbox{ and}\\
\mathfrak{g}_{\pm(e_i+e_j)}&=&\mathbb{R}\mathbf{j}(E_{i,j}+E_{j,i})+
\mathbb{R}\mathbf{k}(E_{i,j}+E_{j,i}).
\end{eqnarray*}
\smallskip

(4) The case $\mathfrak{g}=\mathfrak{d}_n=\mathfrak{so}(2n)$ for $n>3$. The
Cartan subalgebra $\mathfrak{t}$ is identified with
$\mathbb{R}^n$ with the standard orthonormal basis $\{e_1,\ldots,e_n\}$.
The root system $\Delta$ is
\begin{equation}\label{root-system-D-n}
\{\pm e_i\pm e_j \mid 1\leqq i<j\leqq n\}.
\end{equation}
In matrices, we have formulas for the $e_i$ and for the root planes for
$e_i\pm e_j$ similar to those in the case of $\mathfrak{b}_n$, i.e.
\begin{eqnarray*}
e_i&=&E_{2i-1,2i}-E_{2i,2i-1}, \\
\mathfrak{g}_{\pm(e_i-e_j)}&=&\mathbb{R}(E_{2i-1,2j-1}+E_{2i,2j}
-E_{2j-1,2i-1}-E_{2j,2i})\\
&&\phantom{X}+\mathbb{R}(E_{2i-1,2j}-E_{2i,2j-1}+E_{2j-1,2i}-E_{2j,2i-1}),
\mbox{ and}\\
\mathfrak{g}_{\pm(e_i+e_j)}&=&
\mathbb{R}(E_{2i-1,2j-1}-E_{2i,2j}-E_{2j-1,2i-1}+E_{2j,2i})\\
&&\phantom{X}+\mathbb{R}(E_{2i-1,2j}+E_{2i,2j-1}-E_{2j-1,2i}-E_{2j,2i-1}).
\end{eqnarray*}
\smallskip

(5) The case $\mathfrak{g}=\mathfrak{e}_6$.
The Cartan subalgebra $\mathfrak{t}$ can be isometrically identified with
$\mathbb{R}^6$ with the standard orthonormal basis $\{e_1,\ldots,e_6\}$.
The root system is
\begin{equation}\label{root-system-E-6}
\{\pm e_i\pm e_j \mid 1\leqq i<j\leqq 5\}
\cup \{\pm\tfrac12 e_1\pm\cdots\pm\tfrac12 e_5\pm\tfrac{\sqrt{3}}{2}e_6
	\mbox{ with odd number of +'s}\}.
\end{equation}
It contains a root system of type
$\mathfrak{d}_5$.
\smallskip

(6) The case $\mathfrak{g}=\mathfrak{e}_7$. The Cartan subalgebra
can be isometrically identified with
$\mathbb{R}^7$ with the standard orthonormal basis $\{e_1,\ldots,e_7\}$.
The root system is
\begin{eqnarray}\label{root-system-E-7}
& &\{\pm e_i\pm e_j \mid 1\leqq i<j<7\} \cup \{\pm\sqrt{2}e_7;
\tfrac12(\pm e_1\pm\cdots\pm e_6\pm\sqrt{2}e_7)\nonumber\\
& &\mbox{  with an odd number of plus signs among the first six coefficients}\}.
\end{eqnarray}
It contains a root system of $\mathfrak{d}_6$.
\smallskip

(7) The case $\mathfrak{g}=\mathfrak{e}_8$. The Cartan subalgebra
can be isometrically identified with
$\mathbb{R}^8$ with the standard orthonormal basis $\{e_1,\ldots,e_8\}$.
The root system $\Delta$ is
\begin{eqnarray}\label{root-system-E-8}
&&\{\pm e_i\pm e_j \mid 1\leqq i<j\leqq 8\} \cup\nonumber\\
&&\{\tfrac12(\pm e_1\pm\cdots\pm e_8) \mbox{ with an even number of +'s}\}.
\end{eqnarray}
It contains a root system of $\mathfrak{d}_8$.
\smallskip

(8) The case $\mathfrak{g}=\mathfrak{f}_4$. The Cartan subalgebra is
isometrically identified with $\mathbb{R}^4$ with the standard orthonormal
basis $\{e_1,\ldots,e_4\}$. The root system is
\begin{eqnarray}\label{root-system-F-4}
\{\pm e_i \mid 1\leqq i\leqq 4\} \cup \{\pm e_i\pm e_j \mid 1\leqq i<j\leqq 4\}
	\cup \{\tfrac12(\pm e_1\pm\cdots\pm e_4)\}.
\end{eqnarray}
It contains the root system of $\mathfrak{b}_4$.
\smallskip

(9) The case $\mathfrak{g}=\mathfrak{g}_2$. The Cartan subalgebra
is isometrically identified with $\mathbb{R}^2$ with the standard
orthonormal basis $\{e_1,e_2\}$. The root system $\Delta$ is
\begin{eqnarray}\label{root-system-G-2}
\{(\pm\sqrt{3},0),(\pm\tfrac{\sqrt{3}}{2},\pm\tfrac{3}{2}),
(0,\pm 1),(\pm \tfrac{\sqrt{3}}{2},\pm\tfrac{1}{2})\}.
\end{eqnarray}
\smallskip

There are many choices of the orthonormal basis $\{e_1,\ldots,e_n\}$
with respect to which the root systems have the same standard presentations as
above, for example the ones obtained by applying elements of the Weyl group.
In the classical cases this means any permutation of the $e_i$\,,
with any number of sign changes $e_i \mapsto \pm e_i$ in cases
$\mathfrak{b}$ and $\mathfrak{c}$, an even number of sign changes in
case $\mathfrak{d}$.  For type $\mathfrak{d}$ we can also use the outer
automorphism and thus have $e_i \mapsto \pm e_i$
with any number of sign changes.
\smallskip

\section{CK vector fields on compact normal homogeneous spaces}
\setcounter{equation}{0}
Assume $M = G/H$ is a Riemannian normal homogeneous space in which $G$ is a
compact connected simple Lie group, and $H$ is a closed subgroup with
$0<\dim H<\dim G$. We keep all notation of the last section and further
assume there is a nonzero vector $v\in\mathfrak{g}$
that defines a Clifford--Killing vector field on $M$.
The value of $v$ at $\pi(g)$, where $\pi :  G \to G/H$ as usual,
is $\pi_*|_g((L_g)_*(\mathrm{Ad}(g)v))=g_*\pi_*|_\mathbf{e}(\mathrm{Ad}(g)v)$. So the
condition that $v$ defines a nonzero CK vector field on $M$
is that $||\mathrm{pr}_{\mathfrak{m}}(\mathrm{Ad}(g)v)||$ is a
positive constant function of $g$. For the bi-invariant inner product,
$$
||v||^2 =
||\mathrm{Ad}(g)v||^2=||\mathrm{pr}_{\mathfrak{h}}(\mathrm{Ad}(g)v)||^2
+||\mathrm{pr}_{\mathfrak{m}}(\mathrm{Ad}(g)v)||^2
$$
is a constant function of $g\in G$, so the same is true for
$||\mathrm{pr}_{\mathfrak{h}}(\mathrm{Ad}(g)v)||$. Suitably choosing $v$
within its $\mathrm{Ad}(G)$-orbit, we can assume $v\in\mathfrak{t}$
(the standard special Cartan subalgebra given in the last section).
Now $||\mathrm{pr}_{\mathfrak{h}}(\rho(v))||$ and
$||\mathrm{pr}_{\mathfrak{m}}(\rho(v))||$ are constant functions of $\rho$
in the Weyl group.
Because $\mathfrak{g}$ is simple and $v\neq 0$, both the functions
$||\mathrm{pr}_{\mathfrak{h}}(\mathrm{Ad}(g)(v))||$
and $||\mathrm{pr}_{\mathfrak{m}}(\mathrm{Ad}(g))(v)||$ for $g\in G$,
(or the functions $||\mathrm{pr}_{\mathfrak{h}}(\rho(v))||$ and
$||\mathrm{pr}_{\mathfrak{m}}(\rho(v))||$ for $\rho$ in the Weyl group)
are positive constant functions.  From the above observations, it is easy
to prove two special cases of Theorem \ref{main}:
\smallskip

\begin{proposition}\label{easy-no}
Let $G$ be a compact connected simple Lie group and $H$
a closed subgroup with $0<\dim H<\dim G$.
If $\mathfrak{g}=\mathfrak{a}_2$ or $\mathfrak{g}_2$ then there is no
nonzero $v\in\mathfrak{g}$ that defines a CK vector field on the Riemannian
normal homogeneous space $G/H$.
\end{proposition}
\smallskip

\begin{proof}
Consider $\mathfrak{g}=\mathfrak{a}_2$ first. Assume conversely there is
a nonzero CK vector field, defined by the nonzero vector
$v\in \mathfrak{t}$. The subspaces
$\mathfrak{t}\cap\mathfrak{h}$ and $\mathfrak{t}\cap\mathfrak{m}$ are a pair of
orthogonal lines in $\mathfrak{t}$. Denote all different vectors in the Weyl group orbit of $v$
as $v_1=v$, $\ldots$, $v_k$, $k=3$ or $6$, then
$$\sum_{i=1}^k v_i=\sum_{i=1}^k \mathrm{pr}_\mathfrak{h}(v_i)
=\sum_{i=1}^k \mathrm{pr}_\mathfrak{m}(v_i)=0.$$
All the vectors $\mathrm{pr}_{\mathfrak{m}}(v_i)$ have the same nonzero length,
which only have two possible choices in $\mathfrak{t}\cap\mathfrak{m}$. So
we must have $k=6$, and all $v_i$ can be divided into two set, such that, for example,
$\mathrm{pr}_\mathfrak{m}(v_1)=\mathrm{pr}_\mathfrak{m}(v_2)=\mathrm{pr}_\mathfrak{m}(v_3)$ and
$\mathrm{pr}_\mathfrak{m}(v_4)=\mathrm{pr}_\mathfrak{m}(v_5)=\mathrm{pr}_\mathfrak{m}(v_6)$
are opposite to each other. Obviously $v_1+v_2+v_3\neq 0$, so
there are two $v_i$ among them, $v_1$ and $v_2$ for example, such that $v_1=\rho(v_2)$, in which
$\rho$ is the reflection in some root of $\mathfrak{g}$.
Thus $\mathfrak{t}\cap\mathfrak{h}$, containing $v_1-v_2$, is linearly spanned by a root of $\mathfrak{g}$.
Similar argument can also prove $\mathfrak{t}\cap\mathfrak{m}$
is spanned by a root of $\mathfrak{g}$. But for $\mathfrak{a}_2$, there do not exist a pair of orthogonal
roots. This is a contradiction.

The Weyl group of $\mathfrak{g}_2$ contains that of $\mathfrak{a}_2$ as its subgroup,
so the statement for $\mathfrak{g}_2$ also follows immediately the above argument.
\end{proof}

To prove Theorem \ref{main} in general we need some preparation.  Suppose
that $M = G/H$ is a Riemannian normal homogeneous space and $v \in
\mathfrak{g}$ defines a CK vector field on $M$.  If $\psi: G' \to G$ is a
covering group and $H'$ is an open subgroup of $\psi^{-1}(H)$, then
$M' = G'/H'$ is a Riemannian normal homogeneous space and a Riemannian
covering manifold of $M$, and the same $v \in \mathfrak{g}$ defines a CK
vector field on $M'$.  Thus we can always replace $G$ by a covering group.
Similarly we can go down to a certain class of subgroups:
\smallskip

\begin{lemma}\label{trivial-lemma}
Let $M = G/H$ be a Riemannian normal homogeneous space such that $v \in
\mathfrak{g}$ defines a CK vector field on $M$.  Let $G'$ be a
closed subgroup of $G$ whose Lie algebra $\mathfrak{g}'$ satisfies
$\mathfrak{g}'=\mathfrak{g}'\cap\mathfrak{h}+\mathfrak{g}'\cap\mathfrak{m}$.
Let $H'$ be a closed subgroup of $G'$ with Lie algebra
$\mathfrak{h}'=\mathfrak{g}'\cap\mathfrak{h}$.  Then the restriction of the
bi-invariant inner product of $\mathfrak{g}$ to $\mathfrak{g}'$ defines a
Riemannian normal homogeneous metric on $M' = G'/H'$. If $v=v'+v''$ with
$v'\in\mathfrak{g}'$, $\langle v'',\mathfrak{g}'\rangle=0$, and
$[v'',\mathfrak{g}']=0$, then $v'$ defines a CK vector field on
$M' = G'/H'$.
\end{lemma}
\smallskip

\begin{proof}
Because of the decomposition $\mathfrak{g}'=(\mathfrak{g}'\cap\mathfrak{h})
+(\mathfrak{g}'\cap\mathfrak{m})$, we also have
$\mathfrak{g}'^\perp=(\mathfrak{g}'^\perp\cap\mathfrak{h})
+(\mathfrak{g}'^\perp\cap\mathfrak{m})$.
The condition that $v=v'+v''$
defines a CK vector field for the Riemannian normal homogeneous space 
$M=G/H$ implies that
\begin{eqnarray*}
||\mathrm{pr}_\mathfrak{m}(\mathrm{Ad}(g')v)||^2
&=&||\mathrm{pr}_\mathfrak{m}(\mathrm{Ad}(g')v')
+\mathrm{pr}_\mathfrak{m}(\mathrm{Ad}(g')v'')||^2\\
&=&||\mathrm{pr}_\mathfrak{m}(\mathrm{Ad}(g')v')
+\mathrm{pr}_\mathfrak{m}(v'')||^2\\
&=&||\mathrm{pr}_\mathfrak{m}(\mathrm{Ad}(g')v')||^2
+||\mathrm{pr}_\mathfrak{m}(v'')||^2
\end{eqnarray*}
is a constant function for $g'\in G'$. And so does
$||\mathrm{pr}_\mathfrak{m}(\mathrm{Ad}(g')v')||^2$, i.e.
$v'$ defines a CK vector field for the Riemannian normal homogeneous space
$M'=G'/H'$.
\end{proof}
\smallskip

We will frequently use Lemma \ref{trivial-lemma}
to reduce our considerations to smaller groups. 
\smallskip

\begin{lemma}\label{lemma-3-3}
Suppose that $0 \ne v\in\mathfrak{t}$ defines a CK vector field on
$M=G/H$. Then
\begin{description}
\item{\rm (1)} If the Weyl group orbit $W(v)$ contains an orthogonal basis
of $\mathfrak{t}$ then
$$\frac{||\mathrm{pr}_{\mathfrak{h}}(v)||^2}{||v||^2}=
\frac{\dim(\mathfrak{t}\cap\mathfrak{h})}{\dim\mathfrak{t}}.$$
\item{\rm (2)} If {\rm Ad}$(G)v$ contains an orthogonal basis of $\mathfrak{g}$
then
$$\frac{||\mathrm{pr}_{\mathfrak{h}}(v)||^2}{||v||^2}=
\frac{\dim\mathfrak{h}}{\dim\mathfrak{g}}.$$
\end{description}
\end{lemma}
\smallskip

\begin{proof}
(1) For simplicity, we assume $||v||=1$.
Let $\{v_1,\ldots,v_n\} \subset W(v)$ be an orthogonal basis of
$\mathfrak{t}$, and $\{u_1,\ldots,u_h\}$ an orthonormal basis of
$\mathfrak{t}\cap\mathfrak{h}$.  Expand $u_i=\sum_{j=1}^n a_{ij}v_j$\,; then
$\mathrm{pr}_{\mathfrak{h}}(v_i)=\sum_{j=1}^h a_{ij}u_j$ and
$||\mathrm{pr}_{\mathfrak{h}}(v_i)||^2=\sum_{j=1}^h a_{ij}^2$\,.  So
$$
\frac{||\mathrm{pr}_{\mathfrak{h}}(v)||^2}{||v||^2}=
\frac{1}{n}\left (\sum_{i=1}^n||\mathrm{pr}_{\mathfrak{h}}(v_i)||^2\right )
=\frac1n\sum_{i=1}^n\sum_{j=1}^h||a_{ij}||^2=
\frac1n\sum_{j=1}^h\sum_{i=1}^h||a_{ij}||^2=\frac{h}n\,.
$$
That proves the first assertion.  The proof of the second is similar.
\end{proof}
\smallskip

Lemma \ref{lemma-3-3} provides a useful tool when we deal the cases
$\mathfrak{g}=\mathfrak{b}_n$ and $\mathfrak{g}=\mathfrak{d}_n$.
\smallskip

\begin{lemma}\label{lemma-3-1}
Suppose that $0 \ne v\in\mathfrak{t}$ defines a CK vector field on
$M=G/H$. Let $\alpha, \beta \in \Delta(\mathfrak{g},\mathfrak{t})$
such that $\langle\alpha,\beta\rangle=0$ and $\alpha(v) \ne 0 \ne \beta(v)$.
Then
$$
\langle \mathrm{pr}_{\mathfrak{h}}(\alpha),
	\mathrm{pr}_{\mathfrak{h}}(\beta) \rangle=
\langle \mathrm{pr}_{\mathfrak{m}}(\alpha),
	\mathrm{pr}_\mathfrak{m}(\beta)\rangle=0.
$$
\end{lemma}
\smallskip

\begin{proof}
Let the reflections for the roots $\alpha$ and $\beta$ be denoted
$\rho_{\alpha}$ and $\rho_{\beta}$ respectively. Then the four points
$$v_1=\mathrm{pr}_{\mathfrak{h}}v,
v_2=\mathrm{pr}_{\mathfrak{h}}(\rho_{\alpha}(v))
	=\mathrm{pr}_{\mathfrak{h}}(v) -
	 \frac{2\langle v,\alpha\rangle}{\langle\alpha,\alpha\rangle}
 	 \mathrm{pr}_{\mathfrak{h}}(\alpha),$$
$$v_3=\mathrm{pr}_{\mathfrak{h}}(\rho_{\beta}(v))
	=\mathrm{pr}_{\mathfrak{h}}(v)
	 -\frac{2\langle v,\beta\rangle}{\langle\beta,\beta\rangle}
	  \mathrm{pr}_{\mathfrak{h}}(\alpha)$$ and
$$v_4=\mathrm{pr}_{\mathfrak{h}}(\rho_{\beta}\rho_{\alpha}(v))
	= \mathrm{pr}_{\mathfrak{h}}(v)
	 -\frac{2\langle v,\alpha\rangle}{\langle\alpha,\alpha\rangle}
	\mathrm{pr}_{\mathfrak{h}}(\alpha)-
	 \frac{2\langle v,\beta\rangle}{\langle\beta,\beta\rangle}
	 \mathrm{pr}_{\mathfrak{h}}(\alpha)$$
belong to a two dimensional plane and have the same distance from 0.
They are the vertices of a rectangle with adjacent edges parallel to
$\mathrm{pr}_{\mathfrak{h}}(\alpha)$ and $\mathrm{pr}_{\mathfrak{h}}(\beta)$
respectively.  Those edges are  orthogonal, in other words
$\langle\mathrm{pr}_{\mathfrak{h}}(\alpha),
\mathrm{pr}_{\mathfrak{h}}(\beta)\rangle=0$. The other
statement, $\langle\mathrm{pr}_{\mathfrak{m}}(\alpha),
\mathrm{pr}_{\mathfrak{m}}(\beta)\rangle=0$, follows immediately.
\end{proof}
\smallskip

Lemma \ref{lemma-3-1} is the key to our study of the CK vector fields on
the Cartan subalgebra level. The next proposition implies, at least for
classical $\mathfrak{g}$, that a nonzero vector $v$ which defines a
CK vector field must would be very singular.
\smallskip

\begin{proposition}\label{prop-3-2}
Suppose that $\mathfrak{g}$ is classical, i.e.
$\mathfrak{g}=\mathfrak{a}_n$ for $n>0$,
$\mathfrak{b}_n$ for $n>1$,
$\mathfrak{c}_n$ for $n>2$ or
$\mathfrak{d}_n$ for $n>3$.  Suppose that $0 \ne v\in\mathfrak{t}$ defines a
CK vector field on the Riemannian normal homogeneous space $M=G/H$.
Use the standard presentations for the Cartan subalgebra $\mathfrak{t}$ and
root system $\Delta$ in {\rm (\ref{root-system-A-n})--(\ref{root-system-D-n})}.
Then, for a suitable choice of the $e_i$\,, $v$ must be one of the following,
up to multiplication by a nonzero scalar.
\begin{description}
\item{\rm (1)} Let $\mathfrak{g}=\mathfrak{a}_n$ with $n>2$.
 Then $v=n e_1-e_2-\ldots-e_{n+1}$, or $($if  $n$ is odd\,$)$
 $v=e_1+\cdots+e_k-e_{k+1}-\cdots-e_{n+1}$ for $n=2k-1$\,.
\item{\rm (2)} Let $\mathfrak{g}=\mathfrak{b}_n$ or $\mathfrak{c}_n$ with $n>1$,
or let $\mathfrak{g}=\mathfrak{d}_n$ with $n>3$.
 Then $v=e_1$ or $v=e_1+\cdots+e_n$.
\end{description}
\end{proposition}
\smallskip

\begin{proof} (1) Assume $\mathfrak{g}=\mathfrak{a}_n$ with $n>2$.
If the Weyl group orbit of $v$ contains a multiple of
$ne_1-e_2-\cdots-e_{n+1}$\,, then Assertion (1) is proved.  Now suppose
that the Weyl group orbit of $v$ does not contain a multiple of
$ne_1-e_2-\cdots-e_{n+1}$\,. Then
for any orthogonal pair of roots, $\alpha=e_i-e_j$ and
$\beta=e_k-e_l$ with $i$, $j$, $k$ and $l$ distinct,
we can  replace $v$ by a Weyl group conjugate and assume
$v=a_1e_1+\cdots+a_{n+1}e_{n+1}$ where $a_i\neq a_j$ and $a_k\neq a_l$.
Applying Lemma \ref{lemma-3-1}, we have
\begin{equation}\label{proj-coef}
0 = \langle e_i-e_j, e_k-e_l\rangle =
\langle\mathrm{pr}_{\mathfrak{h}}(e_i-e_j),
	\mathrm{pr}_{\mathfrak{h}}(e_k-e_l)\rangle =
\langle\mathrm{pr}_{\mathfrak{h}}(e_i-e_j), e_k-e_l \rangle
\end{equation}
for any $k$ and $l$ such that $i$, $j$, $k$ and $l$ are distinct.
Express $\mathrm{pr}_{\mathfrak{h}}(e_i-e_j) = \sum r_me_m$\,.
Hold $i$ and $j$ fixed, and let $k$ and $\ell$ vary over $\{1,\dots,n+1\}
\setminus \{i,j\}$.  Then (\ref{proj-coef}) shows that
all such $r_k = r_\ell$\,.  Thus we have constants $a$ and $b$ such that
$$
\mathrm{pr}_{\mathfrak{h}}(e_i-e_j)=ae_i+be_j+\tfrac{-a-b}{n-1}
(e_1+\cdots+e_{n+1}-e_i-e_j).
$$
Similarly,
$$
\mathrm{pr}_{\mathfrak{h}}(e_k-e_l)=ce_3+de_4+
\tfrac{-c-d}{n-1}(e_1+\cdots+e_{n+1}-e_k-e_l).
$$
Now (\ref{proj-coef}) tells us
\begin{eqnarray}
\langle\mathrm{pr}_{\mathfrak{h}}(e_i-e_j),\mathrm{pr}_{\mathfrak{h}}(e_k-e_l)
\rangle=\left (\tfrac{n-3}{(n-1)^2}-\tfrac{2}{n-1}\right )(a+b)(c+d)=0,
\end{eqnarray}
i.e. either $a+b=0$ or $c+d=0$.  If $a+b=0$ then
$\mathrm{pr}_{\mathfrak{h}}(e_i-e_j)$ is a multiple of $e_i-e_j$\,, and if
$c+d=0$ then $\mathrm{pr}_{\mathfrak{h}}(e_k - e_l)$ is a multiple of
$e_k - e_l$\,.  If $a+b=0$, so
$\mathrm{pr}_{\mathfrak{h}}(e_i-e_j) = r(e_i - e_j)$,
then $\mathrm{pr}_{\mathfrak{h}}^2(e_i-e_j)
= \mathrm{pr}_{\mathfrak{h}}(e_i-e_j)$ so $r^2 = r$; either
$r = 0$ and $e_i - e_j \in \mathfrak{m}$ or $r = 1$ and
$e_i - e_j \in \mathfrak{h}$\,. Similarly if $c+d=0$ then
either $e_k - e_l \in \mathfrak{m}$ or $e_k - e_l \in \mathfrak{h}$.
So if there is a root $\alpha$ contained neither in $\mathfrak{h}$
nor in $\mathfrak{m}$, then any roots orthogonal to it are contained
either in $\mathfrak{h}$ or in $\mathfrak{m}$.
\smallskip

Suppose that there is a root $\alpha$
contained neither in $\mathfrak{h}$ nor in $\mathfrak{m}$. Applying a
Weyl group element we may assume $\alpha = e_1 - e_2$.  Then any root
$e_i-e_j$, $2<i<j\leqq n+1$, is contained in $\mathfrak{h}$ or
$\mathfrak{m}$, and all such roots must be contained in the same subspace.
Suppose they all belong to $\mathfrak{h}$; the argument will be the
same if they all belong to $\mathfrak{m}$.  Now $e_3-e_4 \in \mathfrak{h}$
and $\langle e_1-e_3,e_3-e_4\rangle\neq 0$ shows
$e_1-e_3\notin\mathfrak{m}$. If $e_1-e_3\notin\mathfrak{h}$,
then by the above argument, $e_2-e_4\in\mathfrak{h}$.
Suitably permuting the $e_i$, we see
$e_i-e_j\in\mathfrak{h}$ for $1<i<j\leqq n+1$, so
$\mathfrak{m}=\mathbb{R}(ne_1-e_2-\cdots-e_{n+1})$.
Recall $v=a_1e_1+\cdots+a_{n+1}e_{n+1}$ with $\sum a_i = 0$.   All
$\mathrm{pr}_{\mathfrak{m}}(\rho(v))$ have the same
length for any $\rho$ in the Weyl group, in other words
$|a_1n - a_2 - \dots - a_{n+1}|= (n+1)|a_1|$ is constant under permutations
of the $a_i$\,.  Thus $n$ is odd, say $n=2k-1$, and after a
suitable permutation of the $e_i$\,, $v$ is a scalar
multiple of $(e_1+\cdots+e_k) - (e_{k+1}+\cdots + e_{n+1})$.
\smallskip

On the other hand suppose that every root $\alpha$ is contained in either
$\mathfrak{h}$ or $\mathfrak{m}$, say
$\Delta = \Delta_\mathfrak{h} \cup \Delta_\mathfrak{m}$.
If $\Delta_\mathfrak{m} = \emptyset$ then $\mathfrak{t} \subset \mathfrak{h}$
and the CK vector field $v$ has a zero, forcing $v = 0$, which contradicts
the hypothesis $v \ne 0$.  If $\Delta_\mathfrak{h} = \emptyset$ then
$\mathfrak{t} \subset \mathfrak{m}$, contradicting our construction of
$\mathfrak{t}$, which starts with a Cartan subalgebra of $\mathfrak{h}$.
Thus $\Delta_\mathfrak{h} \ne \emptyset$ and  $\Delta_\mathfrak{m}\ne \emptyset$.
Because $\langle\Delta_\mathfrak{h},\Delta_\mathfrak{m}\rangle=0$, this is
a contradiction with the fact that $\mathfrak{g}$ is simple. That
completes the proof of Assertion (1).
\smallskip

(2) Assume $\mathfrak{g}=\mathfrak{b}_n$ with $n>1$.   If the Weyl group
orbit $W(v)$ of $v$ contains a multiple of $e_1$ then Assertion (2) is proved
for $\mathfrak{g}=\mathfrak{b}_n$\,.  Now suppose that $W(v)$ does not
contain a multiple of $e_1$.  Express $v=a_1 e_1+\cdots+a_n e_n$; then
at least two of the coefficients $a_i$ are nonzero.  Any two short roots, $e_i$ and $e_j$ with $i\neq j$, are orthogonal to
each other, so we can suitably choose $v$ from its Weyl group orbit such that
$a_i\neq 0$ and $a_j\neq 0$. Applying Lemma \ref{lemma-3-1}, we see
$$
\langle\mathrm{pr}_{\mathfrak{h}}(e_i),\mathrm{pr}_{\mathfrak{h}}(e_j)\rangle=
\langle\mathrm{pr}_{\mathfrak{m}}(e_i),\mathrm{pr}_{\mathfrak{m}}(e_j)\rangle=0
\text{ whenever } i\neq j.
$$
These can only be true when each $e_i$ is contained in either $\mathfrak{t}\cap\mathfrak{h}$
or $\mathfrak{t}\cap\mathfrak{m}$.  Then, suitably permute the
$e_i$\,, we can assume that $\{e_1$, $\ldots$, $e_h\}$ spans
$\mathfrak{t}\cap\mathfrak{h}$, and $\{e_{h+1}. \ldots , e_n\}$ spans
$\mathfrak{t}\cap\mathfrak{m}$.
All  $\mathrm{pr}_{\mathfrak{h}}(\rho(v))$, $\rho$ in the Weyl group,
have the same length. In particular $v=a_1 e_1+\cdots+a_n e_n$
must satisfy $|a_1|=\cdots=|a_n|$. By suitable scalar changes and Weyl group
actions, we have $v=e_1+\cdots+e_n$.  That completes the proof of Assertion (2)
for $\mathfrak{g}=\mathfrak{b}_n$.  The proof for $\mathfrak{g}=\mathfrak{c}_n$
is similar.
\smallskip

Now assume $\mathfrak{g}=\mathfrak{d}_n$ with $n>3$. The root system
of $\mathfrak{d}_n$ contains two subsystems of type $\mathfrak{a}_{n-1}$
whose intersection is of type $\mathfrak{a}_{n-2}$.   If the Weyl group
orbit $W(v)$ contains a scalar multiple of $e_1$ or of
$e_1$ or $e_1+\cdots + e_n$ then Assertion (2) follows.  If it
contains a scalar multiple of $e_1+\cdots + e_{n-1} - e_n$ we apply the
outer automorphism that restricts to $e_1+\cdots + e_{n-1} - e_n \mapsto
e_1+\cdots + e_{n-1} + e_n$\,, and Assertion (2) follows.  Now suppose
that neither of these holds: $W(v)$ contains neither a multiple of $e_1$
nor a multiple of $e_1+\cdots + e_{n-1} \pm e_n$\,.
Then $v=a_1e_1+\cdots+a_ne_n$ has two nonzero coefficients and not all
the $|a_i|$ are equal.  If
$i,j,k \text{ and } l$ are distinct we have $v'=a'_1e_1+\cdots+a'_ne_n \in W(v)$
such that $a'_i\neq \pm a'_j$ and $a'_k\neq a'_l$,
and $v''=a''_1 e_1+\cdots+a''_n e_n \in W(v)$
such that $a''_i\neq \pm a''_j$ and $a''_k\neq -a''_l$.
Apply Lemma \ref{lemma-3-1} to $\alpha=e_i\pm e_j$ and $\beta=e_k\pm e_l$,
or $\alpha=e_i+e_j$ and $\beta=e_i-e_j$,  the result is
$$
\langle\mathrm{pr}_{\mathfrak{h}}(e_i\pm e_j),
	\mathrm{pr}_{\mathfrak{h}}(e_k\pm e_l) \rangle=
\langle\mathrm{pr}_{\mathfrak{m}}(e_i\pm e_j),
	\mathrm{pr}_{\mathfrak{m}}(e_k\pm e_l)\rangle=0
$$
and
$$
\langle\mathrm{pr}_{\mathfrak{h}}(e_i+e_j),
\mathrm{pr}_{\mathfrak{h}}(e_i-e_j)\rangle
=\langle\mathrm{pr}_{\mathfrak{m}}(e_i+e_j),
\mathrm{pr}_{\mathfrak{m}}(e_i-e_j)\rangle=0,
$$
whenever $i$, $j$, $k$ and $l$ are distinct.
This is only possible when the each one of $\pm e_i\pm e_j$ is contained
in $\mathfrak{t}\cap\mathfrak{h}$ or
$\mathfrak{t}\cap\mathfrak{m}$. By an argument similar that used to prove (1),
we have either $\mathfrak{t}\subset\mathfrak{h}$ and the Riemannian normal
homogeneous space $M=G/H$ has no nonzero CK vector field, or
$\mathfrak{t}\cap\mathfrak{h}=0$ which contradicts with our construction of $\mathfrak{t}$.
\end{proof}
\smallskip

Proposition \ref{prop-3-2} is the key step in the proof of Theorem \ref{main}.
It reduces our discussion for each classical $\mathfrak{g}$ to very few
possibilities for the vector $v$. In the next section, we will apply this
proposition to each exceptional $\mathfrak{g}$ and show there does not exist
any nonzero CK vector field in those cases.
\smallskip

\section{Proof of Theorem \ref{main} for $\mathfrak{g}$ Exceptional}
\setcounter{equation}{0}

In this section, we will apply Proposition \ref{prop-3-2} to prove
Theorem \ref{main}
when $\mathfrak{g}$ is a compact exceptional simple Lie algebra.
The proof is a case by case discussion.
\smallskip

(1) The case $\mathfrak{g}=\mathfrak{g}_2$ has already been proven in
Proposition \ref{easy-no}.
\smallskip

(2) Let $\mathfrak{g}=\mathfrak{f}_4$. We use the standard presentation
(\ref{root-system-F-4}) for its root system. Its root system has a
subsystem of type $\mathfrak{b}_4$, which defines a subgroup $W'$ of the
Weyl group $W$.  By the argument for the case of
$\mathfrak{b}_n$ in Proposition \ref{prop-3-2},
if $0 \ne v \in \mathfrak{g}$ defines a CK vector field on $M$, we can
re-scale it and use the $W'$ action and assume that either $v=e_1$ or
$v=\frac12(e_1+\cdots+e_4)$. But those belong to the same orbit for $W$.
Considering $v=\frac12(e_1+\cdots+e_4)$, it follows that each of
$\mathfrak{t}\cap\mathfrak{h}$ and $\mathfrak{t}\cap\mathfrak{m}$ is
linearly spanned by a non-empty subset of $\{e_1,\ldots,e_4\}$. Then use $v = e_1$
in the same Weyl group orbit, $||\mathrm{pr}_{\mathfrak{m}}(\rho(v))||$
varies with $\rho \in W'$, contradicting the CK property of $v$.
We conclude that if $\mathfrak{g}=\mathfrak{f}_4$ then $M = G/H$ has
no nonzero CK vector field.
\smallskip

(3) Let $\mathfrak{g}=\mathfrak{e}_6$.  We use the standard presentation
(\ref{root-system-E-6}) for its root system. Its root system has a
subsystem of type $\mathfrak{d}_5$, which defines subgroup $W'$ of the
Weyl group $W$ of $\mathfrak{g}$. Suppose that
$0 \ne v\in\mathfrak{g}$ defines a CK vector field on the Riemannian
normal homogeneous space $M = G/H$. Using the reflections for roots of the form
$\frac12(\pm e_1\pm\cdots\pm e_5\pm\sqrt{3}e_6)$, we can assume that
$v=a_1e_1+\cdots+a_6e_6$ has three nonzero coefficients among the first
five $a_i$\,. By the argument in the $\mathfrak{d}_n$ case of Proposition
\ref{prop-3-2}, if $1\leqq i<j\leqq 5$ then
$\mathrm{pr}_{\mathfrak{h}}(e_i\pm e_j)$ and
$\mathrm{pr}_{\mathfrak{m}}(e_i\pm e_j)$ are four orthogonal vectors in
$\mathbb{R}e_i+\mathbb{R}e_j+\mathbb{R}e_6$. So if $1\leqq i<j\leqq 5$ then
either $e_i+e_j$ or $e_i-e_j$ is contained in
$\mathfrak{t}\cap\mathfrak{h}$ or in $\mathfrak{t}\cap\mathfrak{m}$, and all
those roots define the same subspace.
We will argue the case where they are in $\mathfrak{t}\cap\mathfrak{h}$;
with very minor modifications our argument also works when they are in
$\mathfrak{t}\cap\mathfrak{m}$.  In that case
there are two possibilities: (1) there exist $i$ and $j$ with $1\leqq i<j<6$
and both $\pm e_i\pm e_j$ contained in $\mathfrak{t}\cap\mathfrak{h}$,
or (2) whenever $1\leqq i<j<6$ either $e_i-e_j$ or $e_1+e_j$ is not
contained in $\mathfrak{t}\cap(\mathfrak{h}\cup\mathfrak{m})$.
\smallskip

Suppose that whenever $1\leqq i<j<6$ either $e_i-e_j$ or $e_i+e_j$ is not
contained in $\mathfrak{t}\cap(\mathfrak{h}\cup\mathfrak{m})$.
By suitably choosing the first five $e_i$\,, in the $\mathfrak{d}_5$
where $W'$ acts, we can assume $e_i-e_j\in\mathfrak{h}$ and
$e_i+e_j\notin\mathfrak{h}\cup\mathfrak{m}$ for $1\leqq i<j<5$.  Let $\rho$
be the reflection in the root $\frac12(e_1-e_2+e_3+e_4+e_5+\sqrt{3}e_6)$.
Denote $e'_i=\rho(e_i)$. Apply the above argument to the new basis
$\{e'_1,\cdots,e'_6\}$.  Then for $1\leqq i<j<6$, either
$e'_i+e'_j$ or $e'_i-e'_j$ is not contained in
$\mathfrak{t}\cap(\mathfrak{h}\cup\mathfrak{m})$.
Because $e'_1+e'_2=e_1+e_2$ is not contained in
$\mathfrak{t}\cap(\mathfrak{h}\cup\mathfrak{m})$, and because
$e'_1-e'_2=\frac12(e_1-e_2-e_3-e_4-e_5-\sqrt{3}e_6)$ is not orthogonal to
$e_1-e_2\in\mathfrak{t}\cap\mathfrak{h}$, we have
\begin{equation}\label{001}
\tfrac{1}{2}(e_1-e_2-e_3-e_4-e_5-\sqrt{3}e_6)\in
	\mathfrak{t}\cap\mathfrak{h}.
\end{equation}
If we use the reflection in the root
$\frac12(e_1-e_2-e_3-e_4+e_5+\sqrt{3}e_6)$, the above argument shows
\begin{equation}\label{002}
\tfrac{1}{2}(e_1-e_2+e_3+e_4-e_5-\sqrt{3}e_6)\in\mathfrak{h}.
\end{equation}
By (\ref{001}) and (\ref{002}), $e_3+e_4\in\mathfrak{t}\cap\mathfrak{h}$,
which contradicts our assumption.
Thus there exist $i$ and $j$ such that $1\leqq i<j<6$
and the $\pm e_i\pm e_j$ are contained in $\mathfrak{t}\cap\mathfrak{h}$,  Then
$\mathfrak{h}=\mathbb{R}e_1+\cdots+\mathbb{R}e_5$ and
$\mathfrak{m}=\mathbb{R}e_6$.
\smallskip

Let $\rho$ be the reflection in a root of the form
$\frac12(\pm e_1\pm e_2 \pm e_3 \pm e_4 \pm e_5\pm\sqrt{3}e_6)$, and denote
$e'_i=\rho(e_i)$ for $1\leqq i\leqq 6$.  The $e'_i$ are another orthonormal
basis of $\mathfrak{t}$ for which the root system $\Delta$ is given by
(\ref{root-system-E-6}). In particular the $\pm e'_i\pm e'_j$, $1\leqq i<j<6$
give a root system of $\mathfrak{d}_5$ in $\Delta$. The above argument also
implies $\mathfrak{m}=\mathbb{R}e'_6$\,.
But $\mathbb{R}e'_6\neq\mathbb{R}e_6$. This is a contradiction.  We conclude
that there is no nonzero vector $v$ that defines a CK vector field for
the Riemannian normal homogeneous space $G/H$.
\smallskip

(3) Use the standard presentation (\ref{root-system-E-7}) for the
root system of $\mathfrak{e}_7$, and apply an argument similar to the one
above for $\mathfrak{e}_6$.  Arguing {\em mutatis mutandis} we see that,
when $\mathfrak{g}=\mathfrak{e}_7$, there is no nonzero vector $v$ that
defines a CK vector field for the Riemannian normal homogeneous space $G/H$.
\smallskip

(4) Let $\mathfrak{g}=\mathfrak{e}_8$.  We use the standard presentation
(\ref{root-system-E-8}) for its root system. Its root system
contains a root system of type $\mathfrak{d}_8$, and the Weyl
group $W'$ of that $\mathfrak{d}_8$ is of course
a subgroup of the Weyl group $W$ of $\mathfrak{g}$.   Suppose that
a nonzero vector $v\in\mathfrak{t}$ defines a CK vector field on the
Riemannian normal homogeneous space $G/H$. The argument for
the case of $\mathfrak{d}_n$ for Proposition \ref{prop-3-2}
can be applied here to show that, up to scalar multiplications and 
the action of
$W'$, either $v=e_1$ or $v=e_1+\cdots+e_7\pm e_8$. In either case, the
reflection in the root $\frac12(e_1+e_2-e_3-\cdots-e_8)$ maps $v$ to another
$v'=a_1e_1+\cdots+a_8e_8$ such that there are at least two nonzero $a_i$s and
not all $|a_i|$s are the same. We have shown
$||\mathrm{pr}_{\mathfrak{h}}(\rho(v))||$ is not a constant function
for all $\rho\in W'$ in the proof of Proposition \ref{prop-3-2}. This
contradicts our assumption on $v$. So there is no nonzero
$v\in\mathfrak{g}$ defining a CK vector field on $M = G/H$.
\smallskip

In summary, we have proved
\smallskip

\begin{proposition}\label{proposition-4-1}
Let $G$ be a compact connected exceptional simple Lie group,
and $H$ a closed subgroup with $0<\dim H<\dim G$. Then there is no nonzero
vector $v \in \mathfrak{g}$ that defines a CK vector field on the
Riemannian normal homogeneous space $M=G/H$.
\end{proposition}

\section{Proof of Theorem \ref{main} for $\mathfrak{g}=\mathfrak{a}_n$}
\setcounter{equation}{0}
In this section  $\mathfrak{g}=\mathfrak{a}_n=\mathfrak{su}(n+1)$ and
$0\ne v\in\mathfrak{t}$ defines a CK vector field on the Riemannian normal
homogeneous space $M = G/H$.  Proposition \ref{prop-3-2} says that, up
to the action of the Weyl group, either $v$ is a multiple of
$ne_1 - e_2 - \dots - e_{n+1}$ or $n+1 = 2k$ and $v$ is a multiple of
$(e_1 + \dots + e_k)- (e_{k+1} + \dots + e_{n+1})$.  However we must see
whether those vectors $v$ actually define CK vector fields. The case
$n = 1$ is trivial. If there are nonzero CK vectors fields, then
$\dim H=0$. The case $n = 2$ has been proven in Proposition \ref{easy-no}.
So we assume $n>2$.
\smallskip

\subsection{The case $n=2k-1$ is odd and $v=(e_1+\cdots+e_k) -
(e_{k+1}+\cdots + e_{n+1})$}
\smallskip

From the argument in Proposition \ref{prop-3-2},
either
$\mathfrak{t}\cap\mathfrak{h}=\mathbb{R}(ne_1-e_2-\cdots-e_{n+1})$, or
$\mathfrak{t}\cap\mathfrak{m}=\mathbb{R}(ne_1-e_2-\cdots-e_{n+1})$, up
to the action of the Weyl group.
\smallskip

Suppose $\mathfrak{t}\cap\mathfrak{h}=\mathbb{R}(ne_1-e_2-\cdots-e_{n+1})$.
Then either $\mathfrak{h}=\mathbb{R}(ne_1-e_2-\cdots-e_{n+1})$ or
$\mathfrak{h}$ is the $\mathfrak{a}_1$ with Cartan subalgebra
$\mathfrak{h} \cap \mathfrak{t}
=\mathbb{R}(ne_1-e_2-\cdots-e_{n+1})$.  In the $\mathfrak{a}_1$
case the root plane $\mathfrak{j}$ of $\mathfrak{h}$ relative to $\mathfrak{t}\cap\mathfrak{h}$
is contained
in $\mathfrak{g}_{\pm(e_1-e_2)}+\cdots+\mathfrak{g}_{\pm(e_1-e_{n+1})}$,
Direct calculation shows $[\mathfrak{j},\mathfrak{j}] \not\subset
\mathfrak{t}\cap\mathfrak{h} + \mathfrak{j}$ so
$\mathfrak{t}\cap\mathfrak{h} + \mathfrak{j}$ is not a Lie algebra.  This
eliminates the $\mathfrak{a}_1$ case.  Thus
$\mathfrak{h}=\mathbb{R}(ne_1-e_2-\cdots-e_{n+1})$.  Let $g \in G$ with
$\Ad(g)v = \sqrt{-1}\,\mathrm{diag}(\left(
                                    \begin{smallmatrix}
                                      0 & 1 \\
                                      1 & 0 \\
                                    \end{smallmatrix}
                                  \right),1,-1,\cdots,1,-1)$.
Then $\langle \Ad(g)v, ne_1-e_2-\cdots-e_{n+1} \rangle = 0$ while
$\langle v, ne_1-e_2-\cdots-e_{n+1} \rangle = n+1$, so $v$ cannot
define a CK vector field on $M = G/H$.  We have proved
$\mathfrak{t}\cap\mathfrak{h} \ne \mathbb{R}(ne_1-e_2-\cdots-e_{n+1})$,
so, up to the action of the Weyl group, we assume
$\mathfrak{t}\cap\mathfrak{m}=\mathbb{R}(ne_1-e_2-\cdots-e_{n+1})$.
\smallskip

If $0 \ne \gamma\in\mathfrak{t}\cap\mathfrak{h}$ then $\gamma \perp e_1$ so
$\dim\widehat{\mathfrak{g}}_{\pm\gamma}=0$ or $2$.
The root planes of $\mathfrak{h}$ are root planes of $\mathfrak{g}$
for roots orthogonal to $e_1$\,. Consider one such,
$\mathfrak{g}_{\pm(e_i-e_j)}$ where $1<i<j\leqq n+1$, which is not a root
plane of $\mathfrak{h}$. Then it is contained in $\mathfrak{m}$.
Permuting the  $e_l$ we may assume $i=3$ and $j=4$. Then we have
$$
v'=\sqrt{-1}\,\mathrm{diag}(1,-1,\left(
                                    \begin{smallmatrix}
                                      0 & 1 \\
                                      1 & 0 \\
                                    \end{smallmatrix}
                                  \right),1,-1,\cdots,1,-1) \in \mathrm{Ad}(G)v.
$$
But $||\mathrm{pr}_{\mathfrak{m}}(v')||>||\mathrm{pr}_{\mathfrak{m}}(v)||$,
which is a contradiction. This proves
$\mathfrak{g}_{\pm(e_i-e_j)}\in\mathfrak{h}$ for $1<i<j\leqq n+1$.
The other root planes of $\mathfrak{g}$ involve $e_1$ in the root, so they are
all contained in $\mathfrak{m}$. In conclusion
$\mathfrak{h}$ is a standard $\mathfrak{su}(n)$ in
$\mathfrak{g}=\mathfrak{su}(n+1)$, and the universal cover of $M=G/H$ is
$S^{2n+1}=\mathrm{SU}(n+1)/\mathrm{SU}(n)$ with $n>1$.
\smallskip

{\em Remark.}
The vector $v=(e_1+\cdots+e_k)-(e_{k+1} + \cdots + e_{n+1})$ defines a CK
vector field 
on the sphere $S^{2n+1} = \mathrm{SO}(2n+2)/\mathrm{SO}(2n+1)$.  However,
the Riemannian normal homogeneous metric on
$S^{2n+1}=\mathrm{SU}(n+1)/\mathrm{SU}(n)$ is not Riemannian symmetric.
The isotropy representation for $S^{2n+1}=\mathrm{SU}(n+1)/\mathrm{SU}(n)$
decomposes $\mathfrak{m}=\mathfrak{m}_0+\mathfrak{m}_1$, in which
$\dim\mathfrak{m}_0=1$ with trivial $\mathrm{Ad}(H)$-action. Let
$\langle\cdot,\cdot\rangle_{\mathrm{bi}}$ be the inner
product on $\mathfrak{m}$ which defines the Riemannian symmetric metric
on $S^{2n+1}$.
The decomposition $\mathfrak{m}=\mathfrak{m}_0+\mathfrak{m}_1$ is orthogonal
for both $\langle\cdot,\cdot\rangle$ and $\langle\cdot,\cdot\rangle_{\mathrm{bi}}$.
By a suitable scalar change, $\langle\cdot,\cdot\rangle_{\mathrm{bi}}$ coincides with
$\langle\cdot,\cdot\rangle$ on $\mathfrak{m}_1$, and differs on $\mathfrak{m}_0$.
If the same $v = (e_1+\cdots+e_k)-(e_{k+1} + \cdots + e_{n+1})$ defines 
a CK vector field on $S^{2n+1}=\mathrm{SU}(n+1)/\mathrm{SU}(n)$
for the Riemannian normal homogeneous metric,
then by the general observations at the beginning of Section 3,
$\mathrm{Ad}(G)v$ is contained in a hyperplane in $\mathfrak{g}$ which is
parallel to $\mathfrak{h}+\mathfrak{m}_1$. That would contradict the fact
$\mathfrak{g}$ is simple and $v$ is nonzero.
\smallskip

\subsection{The case $v=ne_1-e_2-\cdots-e_{n+1}$}

Suppose that some $\mathfrak{g}_{\pm(e_i-e_j)}\subset\mathfrak{m}$.
We may assume $i=1$ and $j=2$.  Then we have
$$
v'=-e_1+ne_2-e_3-\cdots-e_{n+1} \in \mathrm{Ad}(G)v
$$
and
$$
v''=\sqrt{-1}\,\mathrm{diag}\left ( \left(
                     \begin{smallmatrix}
                     (n-1)/2 & (n+1)/2 \\
                     (n+1)/2 & (n-1)/2 \\
                     \end{smallmatrix} \right),-1,\ldots,-1 \right )
\in \mathrm{Ad}(G)v.
$$
Note
$||\mathrm{pr}_{\mathfrak{h}}(v)||=||\mathrm{pr}_{\mathfrak{h}}(v')||
=||\mathrm{pr}_{\mathfrak{h}}(v'')||$. By construction,
$\mathrm{pr}_{\mathfrak{h}}(v)+\mathrm{pr}_{\mathfrak{h}}(v')
=2\mathrm{pr}_{\mathfrak{h}}(v'')$,
so $\mathrm{pr}_{\mathfrak{h}}(v)=\mathrm{pr}_{\mathfrak{h}}(v')$, i.e.
$e_1-e_2\in\mathfrak{t}\cap\mathfrak{m}$. From this argument
$\mathfrak{g}_{\pm(e_i-e_j)}\subset\mathfrak{m}$ only when
$e_i-e_j\in\mathfrak{t}\cap\mathfrak{m}$. In particular,
for any root $e_i-e_j\notin\mathfrak{t}\cap\mathfrak{m}$,
$\mathrm{pr}_{\mathfrak{h}}(e_i-e_j)$ is a root of $\mathfrak{h}$.
Similarly, $\mathfrak{g}_{\pm(e_i-e_j)}\subset\mathfrak{h}$ only
when $e_i-e_j\in\mathfrak{t}\cap\mathfrak{h}$.
\smallskip

It will be convenient to assume
$\mathrm{pr}_{\mathfrak{h}}(e_1+\cdots+e_{n+1})=0$, so that
$\mathrm{pr}_{\mathfrak{h}}$ is defined on the Euclidean space
$\mathbb{R}^{n+1}$ that contains $\mathfrak{t}$.
Denote $e'_i=\mathrm{pr}_{\mathfrak{h}}(e_i)$ for $1\leqq i\leqq n+1$.
Then $e'_i-e'_j$ is $0$ or a root of $\mathfrak{h}$.  The $e'_i$ generate
all the roots of $\mathfrak{h}$.
It follows that $\mathfrak{h}$ is a simple Lie algebra.
Summarizing the above argument, we have
\begin{lemma}\label{lemma-for-case-A-n}
Assume $\mathfrak{g}=\mathfrak{a}_n$ for $n>2$ and suppose that
$v=ne_1-e_2-\cdots-e_{n+1}\in\mathfrak{t}$ defines a CK vector field on
the normal homogeneous space $M=G/H$.   Then
\begin{description}
\item{\rm (1)} If $\mathfrak{g}_{\pm(e_i-e_j)}\subset\mathfrak{m}$, then
$e_i-e_j\in\mathfrak{t}\cap\mathfrak{m}$.
\item{\rm (2)} If $\mathfrak{g}_{\pm(e_i-e_j)}\subset\mathfrak{h}$, then
$e_i-e_j\in\mathfrak{t}\cap\mathfrak{h}$.
\item{\rm (3)} Denote $e'_i=\mathrm{pr}_{\mathfrak{h}}(e_i)$ for
$1\leqq i\leqq n+1$. Then the root system of $\mathfrak{h}$ is the set of
all nonzero vectors of the form
$e'_i-e'_j$. In particular, $\mathfrak{h}$ is a compact simple Lie algebra.
\end{description}
\end{lemma}
\smallskip

Now consider the case $e'_i-e'_j=e'_k-e'_l\neq 0$, in which $e'_i\neq e'_k$
or $e'_j\neq e'_l$ (both inequalities are satisfied). As we saw, $e'_i-e'_j$ is
a root of $\mathfrak{h}$, so $e'_j\neq e'_k$, for otherwise
$e'_i-e'_l=2(e'_i-e'_j)$ is a root of $\mathfrak{h}$; similarly $e'_i\neq e'_l$.
Then for the roots $\alpha'=e'_i-e'_j$ and $\beta'=e'_i-e'_k$ of
$\mathfrak{h}$, both $\alpha'+\beta'=e'_i-e'_l$ and
$\alpha'-\beta'=e'_k-e'_j$ are roots of $\mathfrak{h}$.
There are only two possibilities for this:
\begin{description}
\item{\rm (1)}
The roots $\alpha'$ and $\beta'$ of $\mathfrak{h}$ are short and have an
angle $\frac{\pi}{3}$ or $\frac{2\pi}{3}$, and
$\mathfrak{h}\cong \mathfrak{g}_2$.
\item{\rm (2)}
The roots $\alpha'$ and $\beta'$ of $\mathfrak{h}$ are short,
$\langle\alpha',\beta'\rangle=0$,
and $\pm\alpha'\pm\beta'$ are long roots.
\end{description}
\smallskip

Based on this observation, we will prove the following lemma.
\begin{lemma}\label{another-lemma-for-case-A-n}
Let $\mathfrak{g}=\mathfrak{a}_n$ with $n>2$ and suppose that
$v=ne_1-e_2-\cdots-e_{n+1}\in\mathfrak{t}$ defines a CK vector field
on the Riemannian normal homogeneous space $M=G/H$.  Denote
$e'_i=\mathrm{pr}_{\mathfrak{h}}(e_i)$, so $e'_1+\cdots+e'_{n+1}=0$.
Then we have the following.
\begin{description}
\item{\rm (1)}
No two $e'_i$ are equal.
\item{\rm (2)}
There exist distinct $i$, $j$, $k$ and $l$ such that $e'_i-e'_j=e'_k-e'_l$.
\item{\rm (3)}
If $i$, $j$, $k$ and $l$ are distinct and satisfy $e'_i-e'_j=e'_k-e'_l$,
then the roots $\alpha'$ and $\beta'$ of $\mathfrak{h}$ are short,
$\langle\alpha',\beta'\rangle=0$,
and $\pm\alpha'\pm\beta'$ are long roots.
\end{description}
\end{lemma}
\smallskip

\begin{proof} (1)
For simplicity, we call $e'_i$ {\em single} if its pre-image
$\mathrm{pr}_{\mathfrak{h}}^{-1}(e'_i)$ consists of a single element $e_i$\,.
If $e'_1$ is not single, and
$\mathrm{pr}_{\mathfrak{h}}^{-1}(e'_1) = \{e_1,\ldots,e_k\}$, $k>1$,
then $e'_1-e'_{i}$ is root of $\mathfrak{h}$ for $i>k$.
Permuting $e_i$ for $i>k$, we can assume $e'_1-e'_{k+1}$ has the
largest length among all the roots of the form $e'_1-e'_i$, and
$\mathrm{pr}_{\mathfrak{h}}^{-1}(e'_{k+1})=
e_{k+1},\ldots,e_{k+l}$, $l>0$. If $e'_1-e'_{k+1}=e'_i-e'_j$ with
$i>k$, then by the observation before this lemma,
either there is a root $e'_1-e'_j$ longer than $e'_1-e'_{k+1}$,
which is a contradiction,
or $\mathfrak{h}=\mathfrak{g}_2$ and the angle between $e'_1-e'_{k+1}$ and
$e'_l-e'_i$ is $\frac{\pi}{3}$ or $\frac{2\pi}{3}$.
\smallskip

We first assume there do not exist $i$ and $j$ such that $i>k$ and
$e'_1-e'_{k+1}=e'_i-e'_j$.  Then
$
\mathfrak{h}_{\pm(e'_1-e'_{k+1})}\subset
\widehat{\mathfrak{g}}_{\pm(e'_1-e'_{k+1})}=\sum_{i=1}^k\sum_{j=1}^l
\mathfrak{g}_{\pm(e_i-e_{k+j})},
$
and direct calculation shows $e'_1-e'_{k+1}\in\mathbb{R}(e_1-e_2)+\cdots+
\mathbb{R}(e_{k+l-1}-e_{k+l})$.
Let $G'$ be the closed subgroup of $G$ with its algebra
$$\mathfrak{g}'=\sum_{1\leq i<k+l}\mathbb{R}(e_i-e_{i+1})
+\sum_{1\leq i<j\leq k+l}\mathfrak{g}_{\pm(e_i-e_j)}.$$  It is the standard
$\mathfrak{su}(k+l)$ in $\mathfrak{g}=\mathfrak{su}(n+1)$ corresponding to
the $(k+l)\times(k+l)$-block at the upper left corner. Let $H'$ be the
connected component of $G'\cap H$; its Lie algebra is
$\mathfrak{h}'=\mathfrak{g}'\cap\mathfrak{h}$.
The restriction of the bi-invariant inner product of $\mathfrak{g}$ to
$\mathfrak{g}'$
defines a Riemannian normal homogeneous space $G'/H'$. Then we have the
orthogonal decomposition $\mathfrak{g}'=\mathfrak{h}'+\mathfrak{m}'$,
which coincides with
$\mathfrak{g}'=\mathfrak{g}'\cap\mathfrak{h}+\mathfrak{g}'\cap\mathfrak{m}$.
Notice
$v=ne_1-e_2-\cdots-e_{n+1}$ can be decomposed as a sum of
$$
v'=\tfrac{n+1}{k+l}((k+l-1)e_1-e_2-\cdots-e_{k+l})\in\mathfrak{g}'
$$
and
$$
v''=(\tfrac{n+1}{k+l}-1)(e_1+\cdots+e_{k+l})-(e_{k+l+1}+\cdots+e_{n+1})\in
\mathfrak{c}_{\mathfrak{g}}(\mathfrak{g}'),
$$
with $\langle v'',\mathfrak{g}'\rangle=0$, so by Lemma \ref{trivial-lemma},
$v'=(k+l-1)e_1-e_2-\cdots-e_{k+l}$
defines a CK vector field on the Riemannian normal homogeneous space $G'/H'$.
The subalgebra $\mathfrak{h}'$ is isomorphic to $\mathfrak{a}_1$, which can
be assumed to be linearly spanned by
$$
u_1=\sqrt{-1}\left(
                 \begin{smallmatrix}
                   a I_k & 0 \\
                   0 & b I_l \\
                 \end{smallmatrix}
               \right), u_2=\left(
                              \begin{smallmatrix}
                                0 & A \\
                                -A^* & 0 \\
                              \end{smallmatrix}
                            \right),\mbox{ and } u_3=\sqrt{-1}\left(
                       \begin{smallmatrix}
                         0 & A \\
                         A^* & 0 \\
                       \end{smallmatrix}
                     \right),
$$
in which $ak+bl=0$, $I_k$ and $I_l$ are $k\times k$ and $l\times l$ identity
matrices respectively,
and $A$ is a $k\times l$-complex matrix.
Direct calculation for the condition $[u_2,u_3]\subset\mathbb{R}u_1$ indicates
$a=-b$, $k=l$, and by suitable scalar changes and unitary conjugations, we can
assume $A=I_k$. The orbit $\mathrm{Ad}(G')v'$ contains
$$
v'_1=\sqrt{-1}((k-1)(E_{1,1}+E_{k+1,k+1})+k(E_{1,k+1}+E_{k+1,1})-
\sum_{1\neq i\neq k+1}E_{i,i})
$$
and
$$
v'_2=\sqrt{-1}((k-1)(E_{1,1}+E_{2k,2k})+k(E_{1,2k}+E_{2k,1})-
\sum_{1\neq i\neq 2k}E_{i,i}).
$$
Thus $||\mathrm{pr}_{\mathfrak{h}'}(v'_1)||>
||\mathrm{pr}_{\mathfrak{h}'}(v'_2)||=0$, which is a contradiction.
We conclude that there do exist  $i$ and $j$ such that $i>k$ and
$e'_1-e'_{k+1}=e'_i-e'_j$.  Note then that $i \ne j$.
\smallskip

Now there exist $i \ne j$ such that $i>k$ and
$e'_1-e'_{k+1}=e'_i-e'_j$. In this case $\mathfrak{h}=\mathfrak{g}_2$.
For simplicity, we can suitably permute the $e_i$( not assuming $e'_1=\cdots=e'_k$ any more),
such that we have $e'_1-e'_2=e'_3-e'_4$, the angle between
the short roots $e'_1-e'_2=e'_3-e'_4$ and $e'_1-e'_3=e'_2-e'_4$ of
$\mathfrak{h}$ is $\frac{\pi}{3}$ or $\frac{2\pi}{3}$, and both $e'_1-e'_4$
and $e'_2-e'_3$ are roots of $\mathfrak{h}$ such that one is long and
the other is short. Assume $e'_2-e'_3$ is the
short root for example; assuming $e'_1-e'_4$ to be the short root introduces
only very minor changes in the following argument.
\smallskip

If $e'_2-e'_3=e'_p-e'_q$,
such that $e'_p\neq e'_2$ or $e'_q\neq e'_3$, then either $e'_p$ or $e'_q$
must be different from the $e'_r$s with $1\leqq r\leqq 4$. The short root
$e'_2-e'_p=e'_3-e'_q$ must be one of $\pm(e'_1-e'_2)=\pm(e'_3-e'_4)$
or $\pm(e'_1-e'_3)=\pm(e'_2-e'_4)$.
Then the same $e'_2$ or $e'_3$ appears in different presentations of
a short root of $\mathfrak{h}$, this contradicts our earlier observation.
So there do not exist $p$ and $q$ such that $e'_2-e'_3=e'_p-e'_q$
with either $e'_p\neq e'_2$ or $e'_q\neq e'_3$. Thus
$$
\widehat{\mathfrak{g}}_{\pm(e'_2-e'_3)}={\sum}_{e'_p=e'_2,e'_q=e'_3}\,\,
	\mathfrak{g}_{\pm(e_p-e_q)},
$$
Arguing as above, we see that both $e'_2$ and $e'_3$ are single. We can also get
$e'_2-e'_3\neq e_2-e_3$, otherwise $e'_2-e'_3$ reaches the maximal possible length
of $\mathfrak{h}$, which must be a long root, but we have assumed it is a short root.
By Lemma \ref{lemma-for-case-A-n},
$\mathfrak{h}_{\pm(e'_2-e'_3)}$ is not a root plane of $\mathfrak{g}$, i.e.
$\widehat{\mathfrak{g}}_{\pm(e'_2-e'_3)}=\sum_{e'_p=e'_2,e'_q=e'_3}\mathfrak{g}_{\pm(e_p-e_q)}$
has dimension bigger than 2. So $e'_2$ and $e'_3$ cannot both be single.
This is a contradiction.  Assertion (1) of
Lemma \ref{another-lemma-for-case-A-n} is proved.
\smallskip

(2) If there do not exist distinct indices $i$, $j$, $k$ and $l$ such that
$e'_i-e'_j=e'_k-e'_l$, Then by Lemma \ref{another-lemma-for-case-A-n}(1)
and Lemma \ref{lemma-for-case-A-n}, each root $e_i-e_j$
is either contained in $\mathfrak{t}\cap\mathfrak{h}$ or
$\mathfrak{t}\cap\mathfrak{m}$.  That is only possible when
$\mathfrak{t}\subset\mathfrak{m}$ or
$\mathfrak{t}\subset\mathfrak{h}$, which we have seen is not the case.
\smallskip

(3) follows from the argument of Lemma \ref{another-lemma-for-case-A-n}(1)
which shows that $\mathfrak{h} \not\cong \mathfrak{g}_2$.
\end{proof}
\smallskip

Now we determine $\mathfrak{h}$ by the following lemma.
\begin{lemma}
Let $\mathfrak{g}=\mathfrak{a}_n$ with $n>2$ and suppose that
$v=ne_1-e_2-\cdots-e_{n+1}\in\mathfrak{t}$ defines a CK vector field
on the Riemannian normal homogeneous space $M=G/H$. Keep all relevant notations. Then
for any
$1\leqq i\leqq n+1$, there is a unique
$j$ such that
$e_i-e_j=e'_i-e'_j\in\mathfrak{h}$ is a long root of $\mathfrak{h}$. 
Furthermore, $n$ is odd.
\end{lemma}
\begin{proof}
Since
$\mathfrak{t}$ is not contained in $\mathfrak{h}$, for any $e_i$,
there is another $e_l$ such that $e_i-e_l\notin\mathfrak{h}$. Then the root
$e'_i-e'_l$ of $\mathfrak{h}$ is not a root of $\mathfrak{g}$. By
Lemma \ref{lemma-for-case-A-n}, $\mathfrak{h}_{\pm(e'_i-e'_l)}$ is not
a root plane of $\mathfrak{g}$ and then
$\dim\widehat{\mathfrak{g}}_{\pm(e'_i-e'_l)}>2$. So
we have $e'_i-e'_l=e'_k-e'_j$ with $k\neq i$ and $l\neq j$. By an earlier
observation, and Lemma \ref{another-lemma-for-case-A-n}(3),
$e'_i-e'_j$ is a long root of $\mathfrak{h}$. There do not exist $p$ and $q$
such that $p\neq i$, $q\neq j$ and $e'_i-e'_j=e'_p-e'_q$.
So $\mathfrak{h}_{\pm(e'_i-e'_j)} =\widehat{\mathfrak{g}}_{\pm(e'_i-e'_j)}$
is a root plane of $\mathfrak{g}$. By Lemma \ref{lemma-for-case-A-n}(2),
$e'_i-e'_j=e_i-e_j\in\mathfrak{t}\cap\mathfrak{h}$. There does not exist
another index $p$ such that
$p\neq j$ and $e_i-e_p\in\mathfrak{h}$, because $e_i-e_p$ is not orthogonal
to $e_i+e_j-e_l-e_k\in\mathfrak{m}$. Obviously the map from $i$ to $j\neq i$
maps $j$ back to $i$. It follows immediately that $n$ must be odd.
\end{proof}
\smallskip

After a suitable permutation of the $e_i$ we can assume $\mathfrak{h}$
contains $e_1-e_{k+1}$, $e_2-e_{k+2}$, $\ldots$, $e_{k}-e_{n+1}$, where 
$k=(n+1)/2$,
and at the same time $\mathfrak{m}$ contains
$e_1+e_{k+1}-e_2-e_{k+2}$, $e_2+e_{k+2}-e_3-e_{k+3}$,
$\ldots$, $e_{k-1}+e_{2k-1}-e_k-e_n$. Those are bases for the subspaces
$\mathfrak{h}$ and $\mathfrak{m}$. The root system of $\mathfrak{h}$ is
$$
\{\pm(e_i-e_{i+k})\mid 1\leqq i\leqq k\}\cup
	\{\pm(e'_i-e'_j)\mid 1\leqq i<j\leqq k\}\cup
	\{\pm(e'_i-e'_{j+k})\mid  1\leqq i<j\leqq k\}.
$$
Thus $\mathfrak{h}$ is isomorphic to $\mathfrak{c}_k=\mathfrak{sp}(k)$.
For the root planes, we have
\begin{eqnarray*}
\mathfrak{h}_{\pm(e_i-e_{i+k})}&=&
\widehat{\mathfrak{g}}_{\pm(e_i-e_{i+k})}=\mathfrak{g}_{\pm(e_i-e_{i+k})},\\
\widehat{\mathfrak{g}}_{\pm(e'_i-e'_j)}&=&\mathfrak{g}_{\pm(e_i-e_j)}+
 \mathfrak{g}_{\pm(e_{i+k}-e_{j+k})}\text{ for } 1\leqq i<j\leqq k\mbox{ and}\\
\widehat{\mathfrak{g}}_{\pm(e'_i-e'_{j+k})}&=&
\mathfrak{g}_{\pm(e_i-e_{j+k})}+\mathfrak{g}_{\pm(e_j-e_{i+k})},\text{ for }
   1\leqq i<j\leqq k.
\end{eqnarray*}
\smallskip

We will see $[\mathfrak{m},\mathfrak{m}]\subset\mathfrak{h}$, i.e. $G/H$ is
Riemannian symmetric.
If $\alpha'$ is a root of $\mathfrak{h}$ such that
$\widehat{\mathfrak{g}}_{\pm\alpha'}\cap\mathfrak{m}\neq 0$,
and $u\in\mathfrak{t}\cap\mathfrak{m}$ then
$\ad(u):\widehat{\mathfrak{g}}_{\pm\alpha'}\rightarrow
\widehat{\mathfrak{g}}_{\pm\alpha'}$ is the same
$\left ( \begin{smallmatrix} 0 & a\\ -a & 0 \end{smallmatrix} \right )$
on the root planes $\mathfrak{g}_{\pm \alpha}$ where $\alpha$ restricts
to $\alpha'$.  Thus $\ad(u):\widehat{\mathfrak{g}}_{\pm\alpha'}\rightarrow
\widehat{\mathfrak{g}}_{\pm\alpha'}$ is a multiple of an isometry.
For generic $u$, $ad(u)$ maps $\widehat{\mathfrak{g}}_{\pm\alpha'}\cap\mathfrak{h}$
onto $\widehat{\mathfrak{g}}_{\pm\alpha'}\cap\mathfrak{m}$, and thus
$\widehat{\mathfrak{g}}_{\pm\alpha'}\cap\mathfrak{m}$
onto $\widehat{\mathfrak{g}}_{\pm\alpha'}\cap\mathfrak{h}$, because it maps
orthocomplement to orthocomplement. There are no root planes of $\mathfrak{g}$
contained in $\mathfrak{m}$, so the above argument proves
\begin{equation}\label{0010}
[\mathfrak{m}\cap\mathfrak{t},\mathfrak{m}]=
[\mathfrak{m}\cap\mathfrak{t},\mathfrak{m}\cap\mathfrak{t}^\perp]\subset\mathfrak{h}.
\end{equation}
In particular $[u,\mathfrak{h}]=\mathfrak{m}\cap\mathfrak{t}^\perp$ and
$[u,\mathfrak{m}]\subset\mathfrak{h}$ for generic
$u\in\mathfrak{t}\cap\mathfrak{m}$.
If $w_1,w_2\in\mathfrak{m}\cap\mathfrak{t}^\perp$ and
$u\in\mathfrak{t}\cap\mathfrak{m}$ is generic, now
$[u,w_1]$, $[u,w_2]\subset \mathfrak{h}$, and
$$
[u,[w_1,w_2]]=[[u,w_1],w_2]+[w_1,[u,w_2]]\subset \mathfrak{m},
$$
so
$$
\langle[w_1,w_2],\mathfrak{m}\cap\mathfrak{t}^\perp\rangle=
\langle[w_1,w_2],[u,\mathfrak{h}]\rangle=
\langle[u,[w_1,w_2]],\mathfrak{h}\rangle=0.
$$
Also we have
$$
\langle[w_1,w_2],u\rangle=\langle w_1,[w_2,u]\rangle=0,
$$
so $\langle[w_1,w_2],\mathfrak{t}\cap\mathfrak{m}\rangle=0$.
Thus
\begin{equation}\label{0011}
[\mathfrak{m}\cap\mathfrak{t}^\perp,\mathfrak{m}\cap\mathfrak{t}^\perp]
	\subset\mathfrak{h}.
\end{equation}
By (\ref{0010}) and (\ref{0011}),
$[\mathfrak{m},\mathfrak{m}]\subset\mathfrak{h}$, in other words
$G/H$ is a symmetric homogeneous space. It is locally Riemannian symmetric as well.
\smallskip

To summarize, we have proved the following proposition.
\begin{proposition}\label{prop-5-3}
Let $G$ be a compact connect simple Lie group of type $\mathfrak{a}_n$ and
$H$ a closed subgroup with $0<\dim H<\dim G$. If there
is a nonzero vector $v\in\mathfrak{g}=\mathrm{Lie}(G)$ that defines a
CK vector field on the Riemannian normal homogeneous space $M=G/H$, then
$M$ is a local Riemannian symmetric space with universal Riemannian covering
space $SU(2k)/Sp(k)$.
\end{proposition}
\smallskip

\section{Proof of Theorem \ref{main} for $\mathfrak{g}=\mathfrak{b}_n$}
\setcounter{equation}{0}
In this section $\mathfrak{g}=\mathfrak{b}_n=\mathfrak{so}(2n+1)$ with $n>1$,
and $0 \ne v\in\mathfrak{t}$ defines a CK vector field on the Riemannian normal
homogeneous space $M=G/H$.  Proposition \ref{prop-3-2} says that, up to
the action of the Weyl group, either $v$ is a multiple of $e_1$ or a
multiple of $e_1 + \dots + e_n$\,.  We must see whether those
vectors $v$ define CK vector fields on $M$.
\smallskip

\subsection{The case $v=e_1+\cdots+e_n$}\label{subsection-b-n-1}
Following the proof of Proposition \ref{prop-3-2} we may suitably permute
the $e_i$ and assume
\begin{eqnarray*}
\mathfrak{h}=\mathbb{R}e_1+\mathbb{R}e_{n-h+2}+\cdots+\mathbb{R}e_n,
	\mbox{ and }
\mathfrak{m}=\mathbb{R}e_{2}+\cdots+\mathbb{R}e_{n-h+1}.
\end{eqnarray*}
Let $G'$ be the connected Lie subgroup of $G$ whose Lie algebra
$\mathfrak{g}'$ is the centralizer of
$\mathbb{R}e_{n-h+2}\cdots+\mathbb{R}e_n$ in $\mathfrak{g}$.
Let $H'$ be the connected component of $G'\cap H$; its Lie algebra
$\mathfrak{h}'=\mathfrak{h}\cap\mathfrak{g}'$.  We have a direct sum
decomposition $\mathfrak{g}'=\mathfrak{g}''\oplus\mathfrak{c}$
in which $\mathfrak{c}$ is the center and
$\mathfrak{g}''=\mathfrak{so}(2n-2h+3)$ corresponding to
the $e_i$ with $1\leqq i\leqq n-h+1$. Here
$\mathfrak{c}=\mathbb{R}e_{n-h+2}+\cdots+\mathbb{R}e_n$.
We also have a direct sum decomposition
$\mathfrak{h}'=\mathfrak{h}''\oplus\mathfrak{c}$, in which either
$\mathfrak{h}'' = \mathbb{R}e_1$ or $\mathfrak{h}''$ is of type
$\mathfrak{a}_1$ with Cartan subalgebra
$\mathfrak{t}\cap\mathfrak{h}''=\mathbb{R}e_1$\,.
Let $G''$ be the closed subgroup in the universal cover of $G'$ with Lie
algebra $\mathfrak{g}''$, and $H''$ the closed connected subgroup
of $G''$ with Lie algebra $\mathfrak{h}''$. The restriction of the bi-invariant
inner product of $\mathfrak{g}$ to $\mathfrak{g}'$ and $\mathfrak{g}''$ defines
locally isometric Riemannian normal homogeneous metrics on $G'/H'$ and
$G''/H''$ respectively. Thus we have orthogonal decompositions
\begin{eqnarray}
\mathfrak{g}'=\mathfrak{h}'+\mathfrak{m}' \mbox{ and }
\mathfrak{g}''=\mathfrak{h}''+\mathfrak{m}''\label{0030}.
\end{eqnarray}
Since $\mathfrak{g}'$ is the centralizer
of a subalgebra of $\mathfrak{h}$, the first decomposition in
(\ref{0030}) coincides with
$\mathfrak{g}'=(\mathfrak{g}'\cap\mathfrak{h})+(\mathfrak{g}'\cap\mathfrak{m})$.
Since $\mathfrak{c}\subset\mathfrak{h}'$ is orthogonal to $\mathfrak{m}'$, i.e.
$\mathfrak{m}'\subset \mathfrak{g}''$,
the second decomposition in (\ref{0030}) coincides with
$$
\mathfrak{g}''=(\mathfrak{g}''\cap\mathfrak{h}')+
	(\mathfrak{g}''\cap\mathfrak{m}') =(\mathfrak{g}''\cap\mathfrak{h})
	+(\mathfrak{g}''\cap\mathfrak{m}).
$$
The vector $v$ can be decomposed as the sum of
$v''=e_1+\cdots+e_{n-h+1}\in\mathfrak{g}''$
and $v_\mathfrak{c}=e_{n-h+2}+\cdots+e_n\in\mathfrak{c}$, which is orthogonal to
$\mathfrak{g}''$. So by Lemma
\ref{trivial-lemma}, $v''$ defines a CK vector field on the
Riemannian normal homogeneous space $G''/H''$.
\smallskip

If $\mathfrak{h}'' = \mathbb{R}e_1$ we can find
$$
v''_1=\mathrm{diag}\left ( 0,\left(
                  \begin{smallmatrix}
                    0 & 0 & 1 & 0 \\
                    0 & 0 & 0 & 1 \\
                    -1 & 0 & 0 & 0 \\
                    0 & -1 & 0 & 0 \\
                  \end{smallmatrix}
                \right),\left(
                          \begin{array}{cc}
                            0 & 1 \\
                            -1 & 0 \\
                          \end{array}
                        \right),\ldots,
                        \left(
                          \begin{array}{cc}
                            0 & 1 \\
                            -1 & 0 \\
                          \end{array}
                        \right)\right )\in \mathfrak{m}''\cap \Ad(G)v''\,.
$$
Then $||\mathrm{pr}_{\mathfrak{h}''}(v'')||>
||\mathrm{pr}_{\mathfrak{h}''}(v''_1)||=0$,
which contradicts the CK property of $v$.
If $\mathfrak{h}''\cong \mathfrak{a}_1$, we use the Weyl group of $G''$
to change it to the standard $\mathfrak{so}(3)\subset\mathfrak{so}(2n-2h+3)$
corresponding the $3\times 3$-block at the upper left corner.
The argument used for the case $\mathfrak{h}''=\mathbb{R}e_1$ also
leads to a contradiction in this case.
\smallskip

\subsection{The case $v=e_1$}
\smallskip

We now consider the case $v=e_1$ in Proposition \ref{prop-3-2}.
Denote $e'_i=\mathrm{pr}_{\mathfrak{h}}(e_i)$ for $1\leqq i\leqq n$; they
all have the same length.  We have $E_{1,2i}-E_{2i,1}$ and
$E_{1,2i+1}-E_{2i+1,1}$ in $\Ad(G)v \cap \mathfrak{g}_{\pm e_i}$\,.
Thus, for $1\leqq i\leqq n$, $e'_i$ is a root of $\mathfrak{h}$, and
$\dim \widehat{\mathfrak{g}}_{\pm e'_i} >2$.
$\Ad(G)v$ also contains $E_{2i,2j}-E_{2j,2i}
\in\mathfrak{g}_{\pm(e_i+e_j)}+\mathfrak{g}_{\pm(e_i-e_j)}$ for
$1\leqq i<j\leqq n$, so either
$e'_i+e'_j$ or $e'_i-e'_j$ is a root of $\mathfrak{h}$ whenever $i\neq j$.
Any root of $\mathfrak{h}$ has form $\pm e'_i$ or $\pm e'_i\pm e'_j$, and
it follows that $\mathfrak{h}$ is a compact simple Lie algebra.
From the standard description of the roots of $\mathfrak{b}_n$,
$\mathfrak{c}_n$, $\mathfrak{f}_4$ and $\mathfrak{g}_2$\,,
the $\pm e'_i$ cannot be long roots, because all roots of $\mathfrak{h}$ are of the form
$\pm e'_i\pm e'_j$ for $i\neq j$.
\smallskip

If $i \neq j$ then $e'_i \neq \pm e'_j$ because that would give a root
$e'_i\pm e'_j = 0 \text{ or } 2e'_j$.  Thus $\{\pm e'_i\}$ is a set of
$2n$ distinct roots of $\mathfrak{h}$.
Since $\dim \widehat{\mathfrak{g}}_{\pm e'_i}>2$, and
$\mathrm{pr}_{\mathfrak{h}}(\pm e_j) \neq e'_i$ for $i \neq j$, there must
be a root $\alpha$ of the form $\pm e_j\pm e_k$, such that
$\mathrm{pr}_{\mathfrak{h}}(\alpha)=e'_i$, i.e. $e'_i=\pm e'_k\pm e'_l$.
\smallskip

Summarizing the above argument, we have the following lemma.

\begin{lemma} Suppose that $\mathfrak{g}=\mathfrak{b}_n$ with $n>1$ and
that $v=e_1\in\mathfrak{t}$ defines a CK vector field on the Riemannian
normal homogeneous space $M=G/H$. Denote
$e'_i=\mathrm{pr}_{\mathfrak{h}}(e_i)$, then we have the following.
\begin{description}
\item{\rm (1)} $\{\pm e'_1,\ldots,\pm e'_n\}$ consists of $2n$ different roots of
$\mathfrak{h}$.  If $1\leqq i<j\leqq n$ then either
$e'_i+e'_j$ or $e'_i-e'_j$ is a root of $\mathfrak{h}$.
\item{\rm (2)} The Lie algebra $\mathfrak{h}$ is compact simple. If
it is isomorphic to $\mathfrak{b}_n$, $\mathfrak{c}_n$, $\mathfrak{f}_4$ or
$\mathfrak{g}_2$ then the $\pm e'_i$ are short roots.
\item{\rm (3)} Any $e'_i$ can be expressed as $e'_i=\pm e'_j\pm e'_k$ in
which $i$, $j$ and $k$ are different from each other.
\end{description}
\end{lemma}
\smallskip

Because from the $\mathrm{Ad}(G)$-orbit of $v$, we can find an orthonormal basis
for $\mathfrak{t}$ as well as an orthonormal basis for $\mathfrak{g}$, by Lemma
\ref{lemma-3-3},
\begin{equation}\label{0040}
2n+1=\frac{\dim\mathfrak{g}}{\dim\mathfrak{t}}=\frac{\dim\mathfrak{h}}{\dim\mathfrak{t\cap\mathfrak{h}}},
\end{equation}
Denote $h=\dim\mathfrak{t}\cap\mathfrak{h}$. Then the right side of
(\ref{0040}) is $h+2$, $2h+1$, $2h+1$, $2h-1$, $13$, $19$,
$31$, $13$ or $7$, respectively, when $\mathfrak{h}$ is isomorphic to
$\mathfrak{a}_h$, $\mathfrak{b}_h$,
$\mathfrak{c}_h$, $\mathfrak{d}_h$, $\mathfrak{e}_6$,
$\mathfrak{e}_7$, $\mathfrak{e}_8$,
$\mathfrak{f}_4$ or $\mathfrak{g}_2$\,. Because $h<n$, $\mathfrak{h}$ can only
be $\mathfrak{e}_7$, $\mathfrak{e}_8$, $\mathfrak{f}_4$ or $\mathfrak{g}_2$.
\smallskip

First consider the cases where $\mathfrak{h}$ is isomorphic to
$\mathfrak{e}_7$, $\mathfrak{e}_8$ or $\mathfrak{f}_4$. If
$e'_i=\pm e'_j\pm e'_k$, then
$e'_j$ and $e'_k$ have an angle $\frac{\pi}{3}$ or $\frac{2\pi}{3}$, so
the corresponding vector $\pm e'_j\pm (-e'_k)$ is not a root of $\mathfrak{h}$.
Based on this observation, we have the following lemma.

\begin{lemma}\label{lemma-5-3}
Assume $\mathfrak{g}=\mathfrak{b}_n$,
$\mathfrak{h}=\mathfrak{e}_7$,
$\mathfrak{e}_8$ or $\mathfrak{f}_4$,
and $v=e_1\in\mathfrak{t}$
defines a CK vector field on the Riemannian
normal homogeneous space $G/H$. Then
\begin{equation}\label{0041}
\frac{||v||^2}{|| \mathrm{pr}_{\mathfrak{h}} (v)||^2}=
  \frac{\dim\mathfrak{t}}{ \dim(\mathfrak{t}\cap\mathfrak{h})}
  =\dim\widehat{\mathfrak{g}}_{\pm e'_i}-1 \geq 3,
\end{equation}
whenever $1\leq i\leq n$.
\end{lemma}
\begin{proof}
Denote $\dim\widehat{\mathfrak{g}}_{\pm e'_i} =2k+2$ and
$\widehat{\mathfrak{g}}_{\pm e'_i} =\sum_{j=1}^k\mathfrak{g}_{\pm\alpha_j}$.
Consider any root $\alpha_j$
of $\mathfrak{g}$ in the above equality, and assume it has the form
 $\alpha_j=\pm e_p\pm e_q$ with $p<q$. Denote $\bar{\alpha}=\pm e_p \mp e_q$.
Because $\mathfrak{h}$ is not isomorphic to $\mathfrak{g}_2$,
and all the $e'_p$, $e'_q$ and $\mathrm{pr}_{\mathfrak{h}}
({\alpha}_j)=\pm e'_p\pm e'_q$ are short roots of
$\mathfrak{h}$, now $\mathrm{pr}_{\mathfrak{h}}
(\bar{\alpha}_j)=\pm e'_p\mp e'_q$ is not a root of $\mathfrak{h}$,
i.e. $\mathfrak{g}_{\pm\bar{\alpha}_j}\subset\mathfrak{m}$.

Let $\mathfrak{v}=\mathfrak{g}_{\pm
e_i}+\sum_{j=1}^k(\mathfrak{g}_{\pm\alpha_j}+\mathfrak{g}_{\pm\bar{\alpha}_j})$,
then $\mathfrak{v}=(\mathfrak{v}\cap\mathfrak{h})+(\mathfrak{v}\cap
\mathfrak{m})$, and $\mathfrak{v}\cap\mathfrak{h}=\mathfrak{h}_{\pm e'_i}$
is real 2-dimensional. Because different roots $\alpha_j$ correspond to different pairs
$\{p,q\}$ for $\alpha_j=\pm e_p\pm e_q$, all roots $\bar{\alpha}_j$
of $\mathfrak{g}$ are also different with each other. So $\mathfrak{v}$
is a real $4k+2$-dimensional linear space. Inside $\mathrm{Ad}(G)v =
\Ad(G)e_1$ we have an orthonormal basis of $\mathfrak{v}$ consisting of
$E_{1,2i}-E_{2i,1}$ and $E_{1,2i+1}-E_{2i+1,1}$ in $\mathfrak{g}_{\pm e_i}$ and
$E_{2p,2q}-E_{2q,2p}$, $E_{2p,2q+1}-E_{2q+1,2q}$, $E_{2p+1,2q}-E_{2q,2p+1}$ and
$E_{2p+1,2q+1}-E_{2q+1,2p+1}$ in each $\mathfrak{g}_{\pm\alpha_j}+
\mathfrak{g}_{\pm\bar{\alpha}_j}$.  By Lemma \ref{subsection-b-n-1}, and
using some arguments as before, we have
$$
\frac{||v||^2}{||\mathrm{pr}_{\mathfrak{h}}||^2}
=\frac{\dim\mathfrak{t}}{\dim(\mathfrak{t}\cap\mathfrak{h})}
=\frac{\dim\mathfrak{v}}{\dim(\mathfrak{v}\cap\mathfrak{h})}=2k+1=
\widehat{\mathfrak{g}}_{\pm e'_i}-1.
$$
\hfill\end{proof}
\smallskip

If $\mathfrak{h}$ is $\mathfrak{e}_7$, $\mathfrak{e}_8$ or
$\mathfrak{f}_4$,
then, by (\ref{0040}), $n=\dim\mathfrak{t}$ is $9$, $15$ or $6$, respectively,
contradicting (\ref{0041}).
\smallskip

Finally we consider the case $\mathfrak{h}=\mathfrak{g}_2$ and show it is
possible.  By (\ref{0040}), we have $n=3$ for this case, with $\pm e_1$,
$\pm e_2$ and $\pm e_3$ corresponding to the three pairs of short roots.
Suitably choosing $e_i$ and applying sign changes $e_i \mapsto -e_i$ as
appropriate, we may assume
$\mathfrak{t}\cap\mathfrak{m}=\mathbb{R}(e_1+e_2+e_3)$, and
the root planes of $\mathfrak{h}$ are
\begin{eqnarray*}
\mathfrak{h}_{\pm(e_i-e_j)}&=&\mathfrak{g}_{\pm(e_i-e_j)},
  \mbox{ for } 1\leqq i<j\leqq 3,\mbox{ and}\\
\mathfrak{h}_{\pm\frac{1}{3}(e_i+e_j-2e_k)}&\subset&
  \widehat{\mathfrak{g}}_{\pm\frac{1}{3}(e_i+e_j-2e_k)}=
  \mathfrak{g}_{\pm(e_i+e_j)} +\mathfrak{g}_{\pm e_k}
  \mbox{ for } \{i,j,k\}=\{1,2,3\}.
\end{eqnarray*}
The subalgebra $\mathfrak{h}$ is uniquely determined up to the action of
$\mathrm{Ad}(G)$.  Since the isotropy subgroup $\mathrm{G}_2$ is transitive
on directions in the tangent space of $S^7=\mathrm{Spin}(7)/\mathrm{G}_2$,
the $\mathrm{Spin}(7)$--invariant Riemannian metric on that space is the
standard constant positive curvature metric.
Now the vector $v=e_1\in\mathfrak{t}$ does defines a CK vector field on
$S^7=\mathrm{Spin(7)/\mathrm{G}_2}$.
\smallskip

It is well known that $v'=e_1+\cdots+e_4$ defines a CK vector field on
the symmetric space $S^7=\mathrm{Spin(8)}/\mathrm{Spin(7)}$ for the
standard imbedding $\mathfrak{so}(7)\hookrightarrow\mathfrak{so}(8)$.
But if we change the setup by a suitable outer automorphism, using
triality, $v'$ can be changed to $v=e_1$\,, which belongs to the
Cartan subalgebra $\mathfrak{t}$ of the standard
$\mathfrak{so}(7)\subset\mathfrak{so}(8)$. Inside
$\mathfrak{so}(8)$ the intersection of the standard $\mathfrak{so}(7)$ and
the isotropic one is just the $\mathfrak{g}_2$. So $v=e_1$ also defines a
CK vector field on the Riemannian symmetric $S^7=\mathrm{Spin(7)/\mathrm{G}_2}$.
\smallskip

In summary, we have the following proposition.
\begin{proposition}\label{prop-S7}
Let $G$ be a compact connected simple Lie group with
$\mathfrak{g}=\mathfrak{b}_n$ where $n>1$. Let $H$ be closed subgroup with
$0<\dim H<\dim G$, such that $G/H$ is a Riemannian normal homogeneous space.
Assume there is a nonzero vector $v\in\mathfrak{g}$ that
defines a CK vector field on $M=G/H$. Then $M=G/H$ is a locally symmetric
Riemannian manifold whose universal Riemannian covering is
$S^7=\mathrm{Spin}(7)/\mathrm{G}_2$.
\end{proposition}
\smallskip

\section{The proof of Theorem \ref{main} when $\mathfrak{g}=\mathfrak{c}_n$}
\setcounter{equation}{0}
\smallskip

In this section we assume that $\mathfrak{g}=\mathfrak{c}_n=\mathfrak{sp}(n)$
with $n>2$, and that $0 \ne v\in\mathfrak{t}$ defines a CK vector field
on the Riemannian normal homogeneous space $M=G/H$.  According to Proposition
\ref{prop-3-2}, we may assume that either $v=e_1+\cdots+e_n$ or $v = e_1$
\smallskip

\subsection{The case $v=e_1+\cdots+e_n$}
\smallskip

After a suitable permutation of the $e_i$ we may assume
$\mathfrak{h}=\mathbb{R}e_1+\mathbb{R}e_{m+2}+\cdots+\mathbb{R}e_n$ and
$\mathfrak{m}=\mathbb{R}e_2+\cdots+\mathbb{R}e_{m+1}$. Consider the closed
subgroup $G'$ of $G$ whose Lie algebra $\mathfrak{g}'$ is the centralizer of
$\mathbb{R}e_{m+2}+\cdots+\mathbb{R}e_n$. Let $H'$ be the identity component
of $G'\cap H$. Then $\mathfrak{g}' = \mathfrak{g}'' \oplus \mathfrak{c}$
where $\mathfrak{g}''=\mathfrak{sp}(m+1)$ corresponds to
$\{e_1,\ldots,e_{m+1}\}$ and
$\mathfrak{c}=\mathbb{R}e_{m+2}+\cdots+\mathbb{R}e_n$ is its center. The
subalgebra $\mathfrak{h}' = \mathfrak{h}'' \oplus \mathfrak{c}$
where either $\mathfrak{h}'' = \mathfrak{h} \cap \mathfrak{g}''=\mathbb{R}e_1$
or $\mathfrak{h}'' \cong \mathfrak{a}_1$ with Cartan subalgebra $\mathbb{R}e_1$
and root plane $\mathbb{R}\mathbf{j}E_{1,1}+\mathbb{R}\mathbf{k}E_{1,1}$.
Let $G''$ be the analytic subgroup of $G'$ for $\mathfrak{g}''$ and let
$H''$ be the analytic subgroup of $G''$ for $\mathfrak{h}''$.
They are closed subgroups.  The restriction of the bi-invariant inner
product of $\mathfrak{g}$ to $\mathfrak{g}'$ and $\mathfrak{g}''$ defines
locally symmetric Riemannian normal homogeneous metrics
on $G'/H'$ and $G''/H''$ respectively. As argued before, the
orthogonal decomposition $\mathfrak{g}''=\mathfrak{h}''+\mathfrak{m}''$
is the same as $\mathfrak{g}''=(\mathfrak{g}''\cap\mathfrak{h})
+(\mathfrak{g}''\cap\mathfrak{m})$.
We can also decompose $v$ as the sum of
$v''=e_1+\cdots+e_{m+1}\in\mathfrak{t}\cap\mathfrak{g}''$ and
$v_{\mathfrak{c}}=e_{m+2}+\cdots+e_n\in\mathfrak{c}$ which is orthogonal to $\mathfrak{g}''$. By
Lemma \ref{trivial-lemma},
$v''$ defines a CK vector field on the Riemannian normal homogeneous space
$G''/H''$.
From the $\mathrm{Ad}(G'')$-orbit of $v''$ we have
$$
v''_1=\mathbf{i}\left(\mathrm{diag}\left (\left(
                 \begin{smallmatrix}
                   0 & 1 \\
                   1 & 0 \\
                 \end{smallmatrix}
               \right),1,\ldots,1\right )\right)\in\mathfrak{m}'',
$$
i.e. $||\mathrm{pr}_{\mathfrak{h}''}(v'')||>
||\mathrm{pr}_{\mathfrak{h}''}(v''_1)||=0$, which is a contradiction.
\smallskip

{\it Remark}. As in the case 5A in Section 5, for $n>2$,
$v=e_1+\cdots+e_n$ defines a CK vector field for the constant
curvature metric on $S^{4n-1}$.  That metric is normal homogeneous for
$S^{4n-1}=\mathrm{SO}(4n)/\mathrm{SO}(4n-1)$,
but it is not normal homogeneous for $S^{4n-1}=\mathrm{Sp}(n)/\mathrm{Sp}(n-1)$.
\smallskip

\subsection{The case $v=e_1$}
\smallskip
Here $v=e_1$ in Proposition \ref{prop-3-2}, and we denote
$e'_i=\mathrm{pr}_{\mathfrak{h}}(e_i)$ for $1\leqq i\leqq n$. The $e'_i$
all have the same length.
\smallskip

The orbit $\mathrm{Ad}(G)v$ contains
$v'=\mathbf{j}E_{i,i}\subset\mathfrak{g}_{\pm 2e_i}$\,, and
$\mathrm{pr}_{\mathfrak{h}}(v')\neq 0 \neq\mathrm{pr}_{\mathfrak{m}}(v')$,
for $1\leqq i\leqq n$.  So $2e'_i$ is a root of $\mathfrak{h}$ and
$\dim\widehat{\mathfrak{g}}_{\pm 2e'_i}>2$, for $1\leqq i\leqq n$.
\smallskip

Suitably choosing $e_i$ with any necessary sign changes $e_i\mapsto -e_i$,
we may assume, for $1\leqq i<j\leqq n$, that the roots $\pm 2e_i$ and
$\pm 2e_j$ project to the same pair of roots of $\mathfrak{h}$ only when
$e'_i=e'_j$. In other words $e'_i+e'_j\neq 0$, for $1\leqq i<j\leqq n$.
If we have a different presentation for the root $2e'_i$, e.g.
$2e'_i=\pm e'_j\pm e'_k$, then the plus signs must be taken and
$e'_j=e'_k=e'_i$.
\smallskip

If $e'_i=e'_j$ for some $i\neq j$, we can permute $e_i$ so that
$\mathrm{pr}_{\mathfrak{h}}^{-1}(e'_1)$ contains $\{e_1$, $\ldots$, $e_k\}$, where $k>1$,
and it does not contain any other $e_i$. So
$$
\widehat{\mathfrak{g}}_{\pm 2e'_1}=\sum_{1\leqq i\leqq k}\mathfrak{g}_{\pm 2e_i}   +\sum_{1\leqq i< j\leqq k} \mathfrak{g}_{\pm (e_i+ e_j)},
$$
in such a way that the root plane $\mathfrak{h}_{\pm e'_1}$ is
linearly generated by two matrices $u$ and $w$ in $\mathfrak{sp}(k)\subset
\mathfrak{sp}(n)$ with nonzero elements only in the upper left $k\times k$
corner.  So $[u,w]$, which is a nonzero multiple of $e'_1$, is represented
by a matrix in $\mathbb{R}e_1+\cdots+\mathbb{R}e_k \in \mathfrak{sp}(k)$.
Since $e_i-e_j\in\mathfrak{m}$ for $1\leqq i<j\leqq k$, we have
$e'_1=\frac{1}{k}(e_1+\cdots+e_k)$. Any root plane
$\mathfrak{g}_{\pm(e_i-e_j)}$, $1\leqq i<j\leqq k$ is contained in
$\mathfrak{m}$.
\smallskip

Let $G'$ be the closed subgroup of $G$ with Lie algebra $\mathfrak{g}'=
\mathfrak{sp}(k)$ corresponding to $\{e_1$, $\ldots$, $e_k\}$, and
$H'$ the identity component of $G'\cap H$.
The Lie algebra $\mathfrak{h}'=\mathfrak{g}'\cap\mathfrak{h}=
\mathbb{R}e'_1+\mathfrak{h}_{\pm e'_1}$.
The restriction of the bi-invariant inner product of $\mathfrak{g}$ to $\mathfrak{g}'$ defines
a Riemannian normal homogeneous space $M' = G'/H'$. The corresponding
orthogonal decomposition $\mathfrak{g}'=\mathfrak{h}'+\mathfrak{m}'$
coincides with
$\mathfrak{g}'=(\mathfrak{g}'\cap\mathfrak{h})+(\mathfrak{g}'\cap\mathfrak{m})$,
$v=e_1\in\mathfrak{g}'$ defines a CK vector field on the
normal homogeneous space $G'/H'$, and $\mathfrak{h}'$ is spanned by
$$
w_1=\mathbf{i}I_{k\times k}\,,\,\, w_2=\mathbf{j}A+\mathbf{k}B,\mbox{ and }
w_3=-\mathbf{j}B+\mathbf{k}A,
$$
where $A$ and $B$ are real symmetric $k\times k$ matrices.
From $w_1\in\mathbb{R}[w_2,w_3]$, we have $AB=BA$ and $A^2+B^2=cI>0$.

By a suitable $\mathrm{Ad}(G')$ conjugation (which is a $\mathrm{SO}(k)$
conjugation on $\mathfrak{sp}(k)$) we diagonalize $A$ and $B$
simultaneously. By a suitable $\mathrm{Ad}(G')$ conjugation and scalar
multiplications, i.e. $\mathrm{Sp}(k)$-conjugation, we may assume
$w_2=\mathbf{j}I_{k\times k}$ and $w_3=\mathbf{k}I_{k\times k}$. Notice
\begin{eqnarray*}
v'_1&=&\mathrm{diag}\left (\left(
                   \begin{smallmatrix}
                     \frac{\sqrt{2}}{2} & \frac{\sqrt{2}}{2} \\
                     -\frac{\sqrt{2}}{2} & \frac{\sqrt{2}}{2} \\
                   \end{smallmatrix}
                 \right),1,\ldots,1\right )\cdot v
			\cdot\mathrm{diag}\left (\left(
                   \begin{smallmatrix}
                     \frac{\sqrt{2}}{2} & -\frac{\sqrt{2}}{2} \\
                     \frac{\sqrt{2}}{2} & \frac{\sqrt{2}}{2} \\
                   \end{smallmatrix}
                 \right),1,\ldots,1\right )\nonumber\\
                 &=&\mathrm{diag}\left (\left(
                 \begin{smallmatrix}
                 \frac12\mathbf{i} & -\frac12\mathbf{i} \\
                 -\frac12\mathbf{i}& \frac12\mathbf{i} \\
                                         \end{smallmatrix}
                                       \right),0,\ldots,0
                 \right ),\mbox{ and}\nonumber\\
v'_2&=&\mathrm{diag}(1,-\mathbf{j},1,\ldots,1)\cdot v'\cdot
\mathrm{diag}(1,\mathbf{j},1,\ldots,1)\nonumber\\
&=&
\mathrm{diag}\left (\left(
                   \begin{smallmatrix}
                     \frac12\mathbf{i} & -\frac12\mathbf{k} \\
                     -\frac12\mathbf{k} &-\frac12\mathbf{i} \\
                   \end{smallmatrix}
                 \right),0,\ldots,0\right )
\end{eqnarray*}
belong to $\mathrm{Ad}(G')v$, but
$||\mathrm{pr}_{\mathfrak{h}'}(v)||>||\mathrm{pr}_{\mathfrak{h}'}(v'')||=0$.
This is a contradiction. So the $e'_i$ must all be distinct
and $\widehat{\mathfrak{g}}_{\pm 2e'_i}=
\mathfrak{g}_{\pm 2e_i}=\mathfrak{h}_{\pm 2e'_i}$ for $1\leqq i\leqq n$.
This contradicts our earlier
conclusion that $\dim\widehat{\mathfrak{g}}_{\pm 2e'_i}>2$.
\smallskip

In summary, we have proved
\begin{proposition}\label{no-cn}
Let $G$ be a compact connected simple Lie group with
$\mathfrak{g}=\mathfrak{c}_n$ and $n>2$.  Let $H$ be closed subgroup with
$0<\dim H<\dim G$ such that $G/H$ is a Riemannian normal homogeneous space.
Then there is no nonzero vector $v\in\mathfrak{g}$ that
defines a CK vector field on $G/H$.
\end{proposition}
\smallskip

\section{The case $\mathfrak{g}=\mathfrak{d}_n$}
\setcounter{equation}{0}
\smallskip

In this section $\mathfrak{g}=\mathfrak{d}_n=\mathfrak{so}(2n)$ with $n>3$,
and $0 \ne v \in \mathfrak{g}$ defines a CK vector field on the Riemannian
normal homogeneous space $G/H$.
\smallskip

\subsection{The case $v=e_1+\cdots+e_n$}
We consider the case $v=e_1+\cdots e_n$ of Proposition \ref{prop-3-2} with
$n > 4$. If $n=4$, we can use the outer automorphism of $\mathfrak{g}$  that
changes $v$ to $e_1$, which will be discussed in the next case.
By an argument similar to that of Proposition \ref{prop-3-2}(2), we show
$$
\langle\mathrm{pr}_{\mathfrak{h}}(\pm e_i\pm e_j),\pm e_k\pm e_l\rangle
=\langle\mathrm{pr}_{\mathfrak{m}}(\pm e_i\pm e_j),\pm e_k\pm e_l\rangle=0,
$$
whenever $i$, $j$, $k$ and $l$ are different indices. Changing $k$ and $l$ arbitrarily, and
taking linear combinations of these two equalities,
$\mathrm{pr}_{\mathfrak{h}}(e_i)$ and $\mathrm{pr}_{\mathfrak{m}}(e_i)$ are
contained in $\mathbb{R}e_i+\mathbb{R}e_j$. Change $j$ as well, we get
$\mathrm{pr}_{\mathfrak{h}}(e_i)$ and $\mathrm{pr}_{\mathfrak{m}}(e_i)$ are contained
by $\cap_{j\neq i}(\mathbb{R}e_i+\mathbb{R}e_j)=\mathbb{R}e_i$. So either
$e_i\in\mathfrak{h}$ or
$e_i\in\mathfrak{m}$ for each $i$.
\smallskip

Let $m = \dim\mathfrak{t}\cap\mathfrak{m}$.  We will prove $m = 1$.
For if $m > 1$ we can suitably permute $e_i$ so that
$\mathfrak{h}\cap\mathfrak{t}
=\mathbb{R}e_1+\mathbb{R}e_{m+2}+\cdots+\mathbb{R}e_n$
and $\mathfrak{m}\cap\mathfrak{t}
=\mathbb{R}e_2+\cdots+\mathbb{R}e_{m+1}$.
Let $\mathfrak{g}'$ be the centralizer of $\mathbb{R}e_{m+2}+\cdots+
\mathbb{R}e_n$ in $\mathfrak{g}$ and $G'$ the analytic subgroup of $G$
for $\mathfrak{g}'$.  Similarly $\mathfrak{h}'=\mathfrak{g}'\cap\mathfrak{h}$
and $H'$ is the corresponding analytic subgroup.
Then $\mathfrak{g}'=\mathfrak{g}''\oplus\mathfrak{c}$ where
$\mathfrak{g}''\cong \mathfrak{so}(2m+2)$ corresponds to
the $(2m+2)\times(2m+2)$-block in the
upper left corner, and where $\mathfrak{c}=\mathbb{R}e_{m+1}+\cdots+\mathbb{R}e_n$
is the center of $\mathfrak{g}'$.  Observe
$\mathfrak{h}'=\mathfrak{h}''\oplus\mathfrak{c}$ where
$\mathfrak{h}''=\mathfrak{g}''\cap\mathfrak{h}'$ is either the
abelian subalgebra $\mathbb{R}e_1$ or is isomorphic to $\mathfrak{a}_1$ with
Cartan subalgebra $\mathbb{R}e_1$.
Let $G''$ be the analytic subgroup of $G'$ with Lie algebra
$\mathfrak{g}''$ and $H''$ the analytic subgroup of $G''$
for $\mathfrak{h}''$. As we argued in Section \ref{subsection-b-n-1},
the restriction of the bi-invariant inner product of $\mathfrak{g}$ to
$\mathfrak{g}''$ defines a Riemannian normal homogeneous metric on $G''/H''$,
such that
the orthogonal decomposition $\mathfrak{g}''=\mathfrak{h}''+\mathfrak{m}''$
coincides with $\mathfrak{g}''=(\mathfrak{g}''\cap\mathfrak{h})+
(\mathfrak{g}''\cap\mathfrak{m})$. For the corresponding decomposition
$v=v''+v_{\mathfrak{c}}$,
$$
v'' = e_1 + \dots + e_{m+1} \in \mathfrak{g}'' \text{ and }
v_{\mathfrak{c}} = e_{m+2} + \dots + e_n \in \mathfrak{c}\subset\mathfrak{g}''^\perp,
$$
$v''$ defines a CK vector field on the normal homogeneous space
$M'' = G''/H''$.
\smallskip

If $\mathfrak{h}''$ is abelian, then we can choose
$$
v''_1=\mathrm{diag}\left (\left(
                     \begin{smallmatrix}
                       0_{2\times 2} & I_{2\times 2}\\
                       -I_{2\times 2} & 0_{2\times 2}
                     \end{smallmatrix}
                   \right),\left(
                             \begin{smallmatrix}
                               0 & 1 \\
                               -1 & 0 \\
                             \end{smallmatrix}
                           \right),\ldots,
                           \left(
                             \begin{smallmatrix}
                               0 & 1 \\
                               -1 & 0 \\
                             \end{smallmatrix}
                           \right)\right )\in\mathfrak{m}''
			\cap \Ad(G'')v''.
$$
Then $||\mathrm{pr}_{\mathfrak{h}''}(v'')||>
||\mathrm{pr}_{\mathfrak{h}''}(v''_1)||=0$,
which is a contradiction. If $\mathfrak{h}''$ is not abelian, then we can
use suitable $\mathrm{Ad}(G')$ action, i.e. $\mathrm{SO}(2m+2)$ conjugation,
to move $\mathfrak{h}''$ to the subalgebra $\mathfrak{so}(3)$ for the
$3\times 3$-block in the upper left  corner. We can choose
$$
v''_2=\mathrm{diag}\left ( \left(
                     \begin{smallmatrix}
                       0_{3\times 3} & I_{3\times 3}\\
                       -I_{3\times 3} & 0_{3\times 3}
                     \end{smallmatrix}
                   \right),\left(
                             \begin{smallmatrix}
                               0 & 1 \\
                               -1 & 0 \\
                             \end{smallmatrix}
                           \right),\ldots,
                           \left(
                             \begin{smallmatrix}
                               0 & 1 \\
                               -1 & 0 \\
                             \end{smallmatrix}
                           \right)\right )\in\mathfrak{m}''
			\cap \Ad(G'')v''.
$$
Then we again have $||\mathrm{pr}_{\mathfrak{h}''}(v'')||>
||\mathrm{pr}_{\mathfrak{h}'}(v''_2)||=0$,
which is a contradiction.  This completes the proof that $m = 1$.
\smallskip

Now we suitably permute $e_i$ so that
$\mathfrak{t}\cap\mathfrak{h}=\mathbb{R}e_1+\cdots+\mathbb{R}e_{n-1}$
and $\mathfrak{t}\cap\mathfrak{m}=\mathbb{R}e_n$. Whenever $1\leqq i<j<n$,
we have vectors
$$
v'=\sum_{1\leqq k\leqq n-1,i\neq k\neq j}(E_{2k-1,2k}-E_{2k,2k-1}) +u'
	\in \Ad(G)v
$$
in which the possibilities for $u'$ are
$$
E_{2i-1,2j-1}+E_{2i,2j}-E_{2j-1,2i-1}-E_{2j,2i} \text{ and }
E_{2i-1,2j}-E_{2i,2j-1}+E_{2j-1,2i}-E_{2j,2i-1}
$$
in $\mathfrak{g}_{\pm(e_i-e_j)}$, and
$$
E_{2i-1,2j-1}-E_{2i,2j}-E_{2j-1,2i-1}+E_{2j,2i} \text{ and }
E_{2i-1,2j}+E_{2i,2j-1}+E_{2j-1,2i}+E_{2j,2i-1}
$$
in $\mathfrak{g}_{\pm(e_i+e_j)}$.
The condition that
$||\mathrm{pr}_{\mathfrak{h}}(v)||=||\mathrm{pr}_{\mathfrak{h}}(v')||$
for each choice of $u'$ indicates each $u'\in\mathfrak{h}$, i.e.
$\mathfrak{g}_{\pm(e_i\pm e_j)}\subset\mathfrak{h}$ for $1\leqq i<j<n$.
A similar argument can also show each $e_i$ is a root of $\mathfrak{h}$,
and $\widehat{\mathfrak{g}}_{\pm e_i}=
\mathfrak{g}_{\pm(e_i+e_n)}+\mathfrak{g}_{\pm(e_i-e_n)}$.  Now,
up to the action of $\mathrm{Ad}(G)$, $\mathfrak{h}$ is uniquely determined,
and is
the standard $\mathfrak{so}(2n-1)$ in $\mathfrak{so}(2n)$. We can also use
similar argument as for the case 5B in Section 5, to prove directly the homogeneous space is symmetric, i.e. $[\mathfrak{m},\mathfrak{m}]\subset\mathfrak{h}$. Then $G/H$
is Riemannian symmetric, covered by the sphere
$S^{2n-1}=\mathrm{SO}(2n)/\mathrm{SO}(2n-1)$ of positive constant curvature.
It is well known fact that $v=e_1+\cdots+e_n$ defines a CK vector field on
it, because its centralizer $\mathrm{U}(n)$ acts transitively on $S^{2n-1}$.
\smallskip

\subsection{The case $v=e_1$}
Finally we consider the case $v=e_1$ in Proposition \ref{prop-3-2}.
Denote $e'_i=\mathrm{pr}_{\mathfrak{h}}(e_i)$ for $1\leqq i\leqq n$.
They all have the same length. Either $e'_i+e'_j$ or $e'_i-e'_j$ is a
root of $\mathfrak{h}$ because $v'=E_{2i-1,2j-1}-E_{2j-1,2i-1}\in
\mathfrak{g}_{\pm e_i+e_j}+\mathfrak{g}_{\pm e_i-e_j}$ belongs to the $\mathrm{Ad}(G)$-orbit of $v$, i.e.
$||\mathrm{pr}_{\mathfrak{h}}(v')||=||\mathrm{pr}_{\mathfrak{h}}(v)||>0$.
\smallskip

{\bf Distinct Roots.}
We first prove that $\{\pm e'_1,\ldots,\pm e'_n\}$ consists of $2n$ distinct roots of
$\mathfrak{h}$.  For if not, then there are indices $i\neq j$ such that
one of $e'_i\pm e'_j=0$.  Then $2e'_i=e'_i\mp e'_j$ is a root of
$\mathfrak{h}$.
If $2e'_i=\pm e'_k\pm e'_l$, then $||e'_i||=||e'_k||=||e'_l||$ tells us
that the pairs $\{\pm e'_i\}$, $\{\pm e'_k\}$ and $\{\pm e'_l\}$
are equal. If $e'_i\neq\pm e'_k$ and $e'_i+e'_k$ is a root of
$\mathfrak{h}$
then, using $\langle 2e'_i,e'_i+e'_k\rangle>0$,
$2e'_i-(e'_i+e'_k)=e'_i-e'_k$ is a root of $\mathfrak{h}$. Similarly
if $e'_i-e'_k$ is a root of $\mathfrak{h}$ then $e'_i+e'_k$ is a root
of $\mathfrak{h}$.  Notice that
$2e'_i$ is a long root and $e'_i\pm e'_k$ are short roots orthogonal to
each other. This can only happen in a subalgebra of type $\mathfrak{c}_2$\,.
Thus $2(e'_i+e'_k)-2e'_i=2e'_k$ is also a long root of
$\mathfrak{h}$ with $\langle e'_i,e'_k\rangle=0$, and there must be an
index $l$ with $e'_k=\pm e'_l$. From this argument all $\pm 2e'_i$s are
long roots of $\mathfrak{h}$. If there are exactly $m$ different pairs
$\{\pm e'_i\}$, then
$\mathfrak{h}$ is isomorphic to $\mathfrak{c}_m$.
\smallskip

Suitably choosing $e_i$, then we can
assume $e'_i+e'_j\ne 0$ for $i\ne j$. Without loss of generality, we may
assume that $\mathrm{pr}^{-1}_{\mathfrak{h}}(e'_1)$
contains $\{e_1$, $\ldots$, $e_k\}$ but no other $e_l$ and does not
contain any $-e_l$\,.
Then
$$
\mathfrak{g}_{\pm (e_i-e_j)}\subset\mathfrak{m},\mbox{ for } 1\leqq i<j\leqq k
\mbox{ and }
\mathfrak{h}_{\pm 2e'_1}\subset\widehat{\mathfrak{g}}_{\pm 2e'_1}=
{\sum}_{1\leqq i<j\leqq k}\mathfrak{g}_{\pm(e_i+e_j)}.
$$
Now $e'_1\subset[\mathfrak{h}_{\pm 2e'_1},\mathfrak{h}_{\pm 2e'_1}]$ is
realized as a matrix in $\mathfrak{so}(2k)$ corresponding the left upper
$(2k\times 2k)$--block, i.e. $e'_1\in\mathbb{R}e_1+\cdots+\mathbb{R}e_k$.
By our assumption, $e_i-e_j\in\mathfrak{t}\cap\mathfrak{m}$
for $1\leqq i<j\leqq k$, so $e'_1= \frac1k(e_1+\cdots+e_k)$.
\smallskip

Let $G'$ be the analytic subgroup of $G$ with Lie algebra $\mathfrak{g}'=
\mathfrak{so}(2k)$ corresponding the  upper left $2k\times 2k$
corner and $H'$ the analytic subgroup for
$\mathfrak{h}'=\mathfrak{g}'\cap\mathfrak{h}=\mathbb{R}e'_1+
\mathfrak{h}_{\pm 2e'_1}$\,. Then $\mathfrak{h}'\cong \mathfrak{a}_1$.
Also,
$$
\mathfrak{g}'\cap\mathfrak{m}=
{\sum}_{1\leqq i< k}\,\,\mathbb{R}(e_{i+1}-e_i)\,\,+\,\,
{\sum}_{1\leqq i<j\leqq k}\,\,\mathfrak{g}_{\pm (e_i-e_j)}.
$$
The restriction to $\mathfrak{g}'$ of the bi-invariant inner product on
$\mathfrak{g}$ defines a Riemannian normal homogeneous metric on $G'/H'$.
The orthogonal decomposition
$\mathfrak{g}'=\mathfrak{h}'+\mathfrak{m}'$ coincides with the decomposition
$\mathfrak{g}'=(\mathfrak{g}'\cap\mathfrak{h})+(\mathfrak{g}'\cap\mathfrak{m})$.
So the vector $v=e_1\in\mathfrak{g}'$ defines a CK vector field on the
Riemannian normal homogeneous space $G'/H'$.
\smallskip

The orbit $\mathrm{Ad}(G')v$ contains an orthonormal basis of $\mathfrak{g}'$
which in turn contains an orthonormal basis for $\mathfrak{t}\cap\mathfrak{g}'$.
By Lemma \ref{lemma-3-3},
$$
2k-1=\frac{\dim\mathfrak{g}'}{\dim(\mathfrak{t}\cap\mathfrak{g}')}=
\frac{\dim\mathfrak{h}'}{\dim(\mathfrak{h}'\cap\mathfrak{t})}=3,
\mbox{ so } k = 2.
$$
\smallskip

Suitably permuting the $e_i$, we can assume the distinct $e'_i$ are given
by $$
\{e'_1=e'_2,e'_3=e'_4,\ldots,e'_{n-1}=e'_{n}\}.
$$
Then
$\dim\widehat{\mathfrak{g}}_{\pm(e'_1\pm e'_3)}=8,$ in which 2 dimensions belong to
$\mathfrak{h}$ and the other 6 dimensions belong to $\mathfrak{m}$.
The $\mathrm{Ad}(G)$--orbit of $v=e_1$ contains an orthonormal basis of
$\mathfrak{t}$ consisting of all the $e_i$\,.  It also contains an
orthonormal basis of $\widehat{\mathfrak{g}}_{\pm(e_1+e_3)}
+\widehat{\mathfrak{g}}_{\pm(e_1-e_3)}$. As in Lemma \ref{lemma-3-3}, we have
$$
\frac{\dim\mathfrak{t}}{\dim(\mathfrak{t}\cap\mathfrak{h})}=
\frac{\dim(\widehat{\mathfrak{g}}_{\pm(e_1+e_3)}+
\widehat{\mathfrak{g}}_{\pm(e_1-e_3)})}{\dim (\mathfrak{h}\cap(\widehat{\mathfrak{g}}_{\pm(e_1+e_3)}+
\widehat{\mathfrak{g}}_{\pm(e_1-e_3)}))},
$$
which is a contradiction because the left side is $2$ and the right side is
$4$. In summary, we have proved
\begin{lemma}
Suppose $\mathfrak{g}=\mathfrak{d}_n$ with $n>3$.  Suppose that
$v=e_1\in\mathfrak{t}$ defines a CK vector field on the Riemannian normal
homogeneous space $G/H$. Denote $e'_i= \mathrm{pr}_{\mathfrak{h}}(e_i)$ for
$1\leqq i\leqq n$, then $\{\pm e'_1,\ldots,\pm e'_n\}$ has cardinality $2n$.
\end{lemma}
\smallskip

Second, we consider the case where one of $e'_i\pm e'_j$ is a
root of $\mathfrak{h}$ and the other is not, for any $1\leqq i<j\leqq n$.
Given a root $\alpha'$ of $\mathfrak{h}$,
we denote $\dim\widehat{\mathfrak{g}}_{\pm\alpha'}=2k_{\alpha'}$ and
$\dim(\mathfrak{t}\cap\mathfrak{h})=h$, and we express
$\widehat{\mathfrak{g}}_{\pm \alpha'}=
\sum_{i=1}^{k_{\alpha'}}\mathfrak{g}_{\pm\alpha_i}$.
By an argument similar to that of Lemma \ref{lemma-5-3}, we see
$$
\frac{n}{h}=\frac{\dim\mathfrak{t}}{\dim(\mathfrak{t}\cap\mathfrak{h})}=
{\dim\widehat{\mathfrak{g}}_{\pm\alpha'}}
=2k_{\alpha'}\,,\mbox{ i.e. } n=2k_{\alpha'}h.
$$
Thus $k = k_{\alpha'}$ is independent of the
choice of $\alpha'$. Denote the number of roots of $\mathfrak{g}$ and
$\mathfrak{h}$ by $|\Delta|$ and $|\Delta'|$, respectively.  Then
$|\Delta|=2k|\Delta'|$. But $|\Delta|=2n(n-1)$, which implies
$|\Delta'|=2h(2kh-1)\geqq 4h^2-2h$.  Calculate $|\Delta'|$ and $4h^2-2h$
for each simple Lie algebra:
$$
\begin{tabular}{|l|l|l|l|l|l|l|l|l|l|}\hline
& $\mathfrak{a}_q$ & $\mathfrak{b}_q$ & $\mathfrak{c}_q$ & $\mathfrak{d}_q$
	& $\mathfrak{g}_2$ & $\mathfrak{f}_4$ & $\mathfrak{e}_6$
	& $\mathfrak{e}_7$ & $\mathfrak{e}_8$ \\
\hline
$|\Delta'|$ & $q(q+1)$ & $2q^2$ & $2q^2$ & $2q(q-1)$ & $12$ & $48$ & $72$ & $126$ & $240$ \\
\hline
$4h^2-2h$ & $4q^2-2q$&$4q^2-2q$&$4q^2-2q$&$4q^2-2q$&$12$&$56$&$132$&$182$&$240$\\
\hline
\end{tabular}
$$
Now for any compact simple Lie algebra
$\mathfrak{h}$ of rank $h$, the number $|\Delta'|$ of all roots satisfies
$|\Delta'|\leqq 4h^2-2h$, with equality if and only if $k=1$ and
$\mathfrak{h}$ is isomorphic to $\mathfrak{a}_1$, $\mathfrak{g}_2$ or
$\mathfrak{e}_8$.
\smallskip

If $\mathfrak{h}$ is not simple, say
$\mathfrak{h}=\mathfrak{h}_1\oplus\cdots\oplus\mathfrak{h}_p\oplus\mathbb{R}^q,
$
let $h_i$ denote the rank, and $\Delta'_i$ the root system, of
$\mathfrak{h}_i$\,.   Then
\begin{equation}\label{9999}
|\Delta'| = \sum_{i=1}^p|\Delta'_i|\leqq \sum_{i=1}^p(4h_i^2-2h_i)
	\leqq 4(\sum_{i=1}^p h_i)^2-2\sum_{i=1}^p h_i\leqq 4h^2-2h.
\end{equation}
If $q > 0$ the last $\leqq$ in (\ref{9999}) is strict, and if $p > 1$ then
the second to last $\leqq$ in (\ref{9999}) is strict.  So in those cases
$|\Delta'| < 4h^2-2h$.  Now $\mathfrak{h}$ is simple.
\smallskip

Since $\mathfrak{h}$ is simple, it is isomorphic to $\mathfrak{a}_1$\,,
$\mathfrak{g}_2$ or $\mathfrak{e}_8$\,,  Here $\mathfrak{a}_1$
would imply $n = 2$, while we are assuming $n > 3$, so $\mathfrak{h}$ is
$\mathfrak{g}_2$ or $\mathfrak{e}_8$\,,  With some sign changes
$e_i\mapsto -e_i$ now every $e'_i - e'_{i+1}$ is a root of $\mathfrak{h}$.
As $\dim\widehat{\mathfrak{g}}_{\pm(e'_i-e'_{i+1})}=2k_{e'_i-e'_{i+1}}=2$,
i.e. $\mathfrak{h}_{\pm(e'_i-e'_{i+1})}=\mathfrak{g}_{\pm(e_i-e_{i+1})}$,
$$
e'_i - e'_{i+1} \in
  [\mathfrak{g}_{\pm(e_i-e_{i+1})},\mathfrak{g}_{\pm(e_i-e_{i+1})}]
  \subset\mathfrak{t}\cap\mathfrak{h} \mbox{ for } 1\leqq i<n.
$$
Thus $\dim(\mathfrak{t}\cap\mathfrak{h})\geqq n-1>h$. This is a contradiction.

Summarizing the above argument,
\begin{lemma}
Assume $\mathfrak{g}=\mathfrak{d}_n$ with $n>3$ and suppose that
$v=e_1\in\mathfrak{t}$ defines a CK vector field on the Riemannian normal
homogeneous space $G/H$. Denote $e'_i=
\mathrm{pr}_{\mathfrak{h}}(e_i)$ for $1\leqq i\leqq n$. Then there exist
$i< j$ such that both $e'_i\pm e'_j$ are roots of $\mathfrak{h}$.
\end{lemma}
\smallskip

{\bf $\mathfrak{h}$ is simple.}
Third, we will prove that $\mathfrak{h}$ is a compact simple Lie algebra.
Suitably permuting $e'_i$ we can assume both $e'_1+e'_2$ and
$e'_1-e'_2$ are roots of $\mathfrak{h}$. They are orthogonal. For each $i>2$,
 either $e'_1+e'_i$ or $e'_1-e'_i$ is a root of $\mathfrak{h}$.
Suitably replace some $e_i$ by its negative; then we can assume $e'_1-e'_i$ is
a root of $\mathfrak{h}$. Then we have $c_i=\pm 1$ for $i>2$ such that
$\langle c_i e'_2,e'_1 - e'_i\rangle\geq 0$. Then
$\langle e'_1+c_i e'_2,e'_1-e'_i\rangle\geq \langle e'_1,e'_1- e'_i\rangle>0$,
in other words $e'_1-e'_i$ is a root for the same simple component of
$\mathfrak{h}$ as $e'_1+c_i e'_2$. The roots
$\{e'_1-e'_i \mid 1<i\leqq n\}\cup \{e'_1+e'_2\}$ of $\mathfrak{h}$ generate
$\mathfrak{t}\cap\mathfrak{h}$, so
$\mathfrak{h}$ is semi-simple.
If $\mathfrak{h}$ is not simple then the above argument shows
$\mathfrak{h}=\mathfrak{h}_1\oplus\mathfrak{h}_2$  where $e'_1-e'_2$ is
a root of $\mathfrak{h}_1$\,, $e'_1+e'_2$ is a root of $\mathfrak{h}_2$\,,
and $\mathfrak{h}_1$ and $\mathfrak{h}_2$ are simple.
\smallskip

Suppose $\mathfrak{h}=\mathfrak{h}_1\oplus\mathfrak{h}_2$ as just described.
Suppose that there are indices $i\ne j$, both $> 2$, $e'_1-e'_i$ is a root of
$\mathfrak{h}_1$, and $e'_1-e'_j$ is a root of $\mathfrak{h}_2$.
Because $e'_2+e'_i=(e'_1+e'_2)-(e'_1-e'_i)
\notin\mathfrak{t}\cap(\mathfrak{h}_1\cup\mathfrak{h}_2)$ is not a root
of $\mathfrak{h}$, $e'_2-e'_i\in\mathfrak{t}\cap\mathfrak{h}_1$
is a root of $\mathfrak{h}_1$. Similarly,
$e'_2+e'_j\in\mathfrak{t}\cap\mathfrak{h}_2$ is a root of $\mathfrak{h}_2$.
Then neither $e'_i+e'_j=(e'_i-e'_2)+(e'_2+e'_j)$
nor $e'_i-e'_j=(e'_i-e'_1)+(e'_1-e'_j)$
is contained in $\mathfrak{t}\cap(\mathfrak{h}_1 \cup \mathfrak{h}_2)$,
so they are not roots of $\mathfrak{h}$. That is a contradiction. So
all $e'_1-e'_i$ for $i>2$ are roots of the same $\mathfrak{h}_1$ or $\mathfrak{h}_2$.
\smallskip

Suitably choose $e_2$ from $\pm e_2$ so that $e'_1-e'_i$ is a root
of $\mathfrak{h}_1$ for $1 < i \leqq n$. It implies $\mathrm{rk}\mathfrak{h}_1=
\mathrm{rk}\mathfrak{h}-1$, and then
$\mathfrak{h}_2\cong\mathfrak{a}_1$ with the only roots $\pm (e'_1+e'_2)$.
There does not exist any root $e'_i+e'_j$ of $\mathfrak{h}$ for $2<i<j\leq n$,
because otherwise it is a root of $\mathfrak{h}_1$, and it implies
$\mathfrak{t}\cap\mathfrak{h}_1=\mathfrak{t}\cap\mathfrak{h}$ which is a contradiction.
As $e'_1 + e'_2 \perp e_1' - e'_i$, $e'_1+e'_2\in\mathfrak{h}$ is
orthogonal to $e_1-e_i=(e'_1-e'_i)+\mathrm{pr}_{\mathfrak{m}}(e_1-e_i)$ as well, for $1 < i \leqq n$.  That implies
$e'_1+e'_2\in\mathbb{R}(e_1+\cdots+e_n)$.
Let $G'$ be the analytic subgroup of $G$ with Lie algebra
$$
\mathfrak{g}'=\sum_{1\leq i<n}\mathbb{R}(e_i-e_{i+1})+
{\sum}_{1\leqq i<j\leqq n}\,\,\mathfrak{g}_{\pm(e_i-e_j)}.
$$
It is the standard $\mathfrak{su}(n-1)$ in $\mathfrak{so}(2n)$.
The identity component $H'$ of $G'\cap H$ has Lie algebra
$\mathfrak{h}'=\mathfrak{h}_1$.
The restriction of the bi-invariant inner product of $\mathfrak{g}$
to $\mathfrak{g}'$ defines a normal homogeneous metric on $G'/H'$.
For $(\mathfrak{g}',\mathfrak{h}')$, we have the decomposition
$$
\mathfrak{g}'=\mathfrak{t}\cap\mathfrak{g}'+
{\sum}_{\gamma\in\mathfrak{t}\,\,\cap\mathfrak{h}_1}
\widehat{\mathfrak{g}}'_{\pm\gamma},
$$
in which $\widehat{\mathfrak{g}}'_{\pm\gamma}$ coincides with
$\widehat{\mathfrak{g}}_{\pm\gamma}$ because
$\langle\gamma,e_1+\cdots+e_n\rangle=0$ whenever $\gamma$ is a root of
$\mathfrak{h}_1$.  The orthogonal decomposition
$\mathfrak{g}'=\mathfrak{h}'+\mathfrak{m}'$ is given by
$\mathfrak{g}'=(\mathfrak{g}'\cap\mathfrak{h})+
	(\mathfrak{g}'\cap\mathfrak{m})$.
Decompose $v=e_1$ as $v' + v_\mathfrak{c}$ where
$
v'=\tfrac{1}{n}((n-1)e_1-e_2-\cdots-e_n)\in\mathfrak{t}\cap\mathfrak{g}'
$
and
$
v_\mathfrak{c}=\tfrac{1}{n}(e_1+\cdots+e_n) \mbox{ with }
[v_{\mathfrak{c}},\mathfrak{g}']=0
$ and $\langle v_\mathfrak{c},\mathfrak{g}'\rangle=0$.
So $v'$ defines a CK vector field on the Riemannian normal homogeneous space
$G'/H'$.
\smallskip

By Proposition \ref{prop-5-3}, $n$ is even, say $n=2k$ with $k\geq 2$,
and conjugation by an element of $\Ad(G')$ carries
$\mathfrak{h}'$ to the standard $\mathfrak{sp}(k)$ in $\mathfrak{su}(n)$.
So $\mathfrak{h}\cong\mathfrak{c}_k\oplus\mathfrak{a}_1$.
Since $\mathrm{Ad}(G)(v)$ contains an orthonormal
basis for $\mathfrak{g}$ as well as an orthonormal basis for $\mathfrak{t}$,
Lemma \ref{lemma-3-3} implies
$$
2n-1
=\frac{\dim\mathfrak{g}}{\dim\mathfrak{t}}
=\frac{\dim\mathfrak{h}}{\dim(\mathfrak{t}\cap\mathfrak{h})}
=\frac{2k^2+k+3}{k+1}\,, \mbox{ not an integer unless } k = 3.
$$
But if $k = 3$ then $2n-1 = 5$ while $\frac{2k^2+k+3}{k+1} = 6$.  So in
any case this is a contradiction. Summarizing the above arguments, we have
proved

\begin{lemma}
Assume $\mathfrak{g}=\mathfrak{d}_n$ with $n>3$ and suppose that
$v=e_1\in\mathfrak{t}$ defines a CK vector field on the Riemannian normal
homogeneous space $G/H$. Then the Lie algebra $\mathfrak{h}$ is simple.
\end{lemma}
\smallskip

Since $\mathfrak{h}$ is simple, the ratio
$r=\frac{||e'_1+e'_2||}{ ||e'_1-e'_2||}$ can only be $(\sqrt{3})^{\pm 1}$,
$(\sqrt{2})^{\pm 1}$ or $1$.  The $e'_i$ cannot all be mutually
orthogonal because that would imply $\mathfrak{t}\subset \mathfrak{h}$, a
contradiction.  If $r \ne 1$, so $\mathfrak{h}$ has two root lengths
and must be $\mathfrak{b}_h$ or $\mathfrak{c}_h$ with 
$h=\dim(\mathfrak{t}\cap\mathfrak{h})$, or $\mathfrak{f}_4$ with $h=4$, or
$\mathfrak{g}_2$ with $h=2$.  If $r = 1$ and we have $e'_i$ and $e'_j$, such
that $i\neq j$ and $\langle e'_i,e'_j\rangle\neq 0$, either $e'_i+e'_j$ or
$e'_i-e'_j$ is a root whose length is different from that of $e'_1\pm e'_2$.
So in that case also $\mathfrak{h}$ is isomorphic to $\mathfrak{b}_h$,
$\mathfrak{c}_h$, $\mathfrak{f}_4$ or $\mathfrak{g}_2$.
\smallskip

As seen earlier, $\mathrm{Ad}(G)(v)$ contains orthonormal bases of
$\mathfrak{t}$ and $\mathfrak{g}$. By Lemma \ref{lemma-3-3},
\begin{equation}\label{0050}
2n-1=\frac{\dim\mathfrak{g}}{\dim\mathfrak{t}}=
\frac{\dim(\mathfrak{h})}{\dim(\mathfrak{t}\cap\mathfrak{h})}.
\end{equation}

{\bf The $\mathfrak{g}_2$ case.}
When $\mathfrak{h}$ is isomorphic to $\mathfrak{g}_2$, the right side of
(\ref{0050}) is $7$, so $n=4$. In this case $e'_i$ and $e'_j$ must have
an angle $\frac{\pi}3$ or $\frac{2\pi}3$, when $i\neq j$.
The only possible choices for all $\pm e'_i$ are $\pm e'_1$, $\pm e'_2$
and the shorter pair among $\pm e'_1\pm e'_2$. One can't have
four different pairs of $\pm e'_i$. This is a contradiction.
\smallskip

{\bf The $\mathfrak{b}_h$ and $\mathfrak{c}_h$ cases.}
When $\mathfrak{h}$ is isomorphic to $\mathfrak{b}_h$ or $\mathfrak{c}_h$,
the right side of (\ref{0050}) is $2h+1$, so $n=h+1$. Note that
$\mathfrak{h}\cap\mathfrak{t}$ is a hyperplane in $\mathfrak{t}$, so
the complement has $2$ components.  With suitable sign changes
$e_i \mapsto -e_i$ we may suppose that
all $e_i$ belong to the same component and they all have the
same projection to $\mathfrak{m}$. So $e_i-e_j\in\mathfrak{h}$ for $i<j$ and
$\mathfrak{m}\cap\mathfrak{t}=\mathbb{R}(e_1+\cdots+e_n)$.
\smallskip

Since $e'_1-e'_2=e_1-e_2$ is a long root of $\mathfrak{h}$, with length
$\sqrt{2}$, $e'_1+e'_2$ can have length $1$ only when $n=4$. In that case
$\pm(e'_1+e'_2)=\mp(e'_3+e'_4)$, $\pm(e'_1+e_3)=\mp(e'_2+e'_4)$ and
$\pm(e'_1+e'_4)=\mp(e'_2+e'_3)$ are the only three possible pairs of short
roots that are orthogonal to each other, so $\mathfrak{h}$ is isomorphic
to $\mathfrak{b}_3$\,, where the root system contains all above shorts roots
and all long roots of the form $\pm(e_i-e_j)$. Up to $\mathrm{Ad}(G)$
conjugacy now $\mathfrak{h}$ is uniquely determined, and it satisfies the
symmetric space condition $[\mathfrak{m},\mathfrak{m}]\subset\mathfrak{h}$.
By a suitable outer automorphism of $\mathfrak{d}_4$, it can be changed back
to the standard $\mathfrak{so}(7)$ inside $\mathfrak{g}=\mathfrak{so}(8)$,
and so $G/H$ is a locally Riemannian symmetric space,
covered by the round $S^7 =
\mathrm{Spin}(8)/\mathrm{Spin}(7)=\mathrm{SO}(8)/\mathrm{SO}(7)$.
At the same time,
the automorphism changes $v=e_1$ to a vector of the form
$\pm e_1\pm\cdots\pm e_4$, which is well known to define a CK vector field
on $S^7$.
\smallskip

{\bf The $\mathfrak{f}_4$ case.}
For the rest of this section, $\mathfrak{h}$ is isomorphic to
$\mathfrak{f}_4$. Then the right side of (\ref{0050}) is $13$, so
$n=7$. For simplicity, we rescale the inner product of $\mathfrak{g}$,
so that $||e'_i||=1$ for $1\leqq i\leqq n$. Then
$\langle e'_1, e'_2\rangle$ must be $0$ or $\pm\frac13$.
\smallskip

Our next step is to prove the case that
$e'_1\pm e'_2$ are roots of $\mathfrak{h}$ with
$\langle e'_1,e'_2\rangle=\pm\frac13$ is impossible. Assume conversely it happens,
then one of $e'_1\pm e'_2$ is short
with length $\frac{2}{\sqrt{3}}$ and the other is long with
length $\frac{2\sqrt{2}}{\sqrt{3}}$. For any indices $i\neq j$, either
$e'_i+e'_j$ or $e'_i-e'_j$ is a root of $\mathfrak{h}$, with length
$\frac{2}{\sqrt{3}}$ or $\frac{2\sqrt{2}}{\sqrt{3}}$. So
$\langle e_i,e_j\rangle=\pm\frac13$ for $1\leqq i<j\leqq n$.
\smallskip

We make appropriate sign changes
$e_i\mapsto -e_i$ so that $\langle e'_1,e'_i\rangle=-\frac13$ for all $i>1$.
So if $e'_1-e'_i$ is a root of $\mathfrak{h}$, it is a long root, and
if $e'_1+e'_i$ is a root of $\mathfrak{h}$, it is a short root.
\smallskip

For any $i>2$, if $e'_1+e'_i$ is a short root of $\mathfrak{h}$,
then $e'_1+e'_i$ has an angle $\frac{\pi}{4}$ with the long root $e'_1-e'_2$,
because
$
\langle e'_1+e'_i,e'_1-e'_2\rangle=1-\langle e'_i,e'_2\rangle>0.
$
So
$$
\langle e'_1+e'_i,e'_1-e'_2\rangle= |e'_1 + e'_i|\cdot|e'_1 - e'_2|
	\cdot \cos\tfrac{\pi}{4} = \tfrac{2}{\sqrt{2}}\cdot
	\tfrac{2\sqrt{2}}{\sqrt{3}}\cdot\tfrac{\sqrt{2}}{2} = \tfrac43\,.
$$
That implies
$\langle e'_2,e'_i\rangle=-\frac13$.
If $e'_1-e'_i$ is a long root of $\mathfrak{h}$,
then $e'_1-e'_i$ has an angle $\frac{\pi}{4}$ with the short root $e'_1+e'_2$,
because
$\langle e'_1-e'_i,e'_1+e'_2\rangle=1-\langle e'_i,e'_2\rangle>0$.
Arguing as above, $\langle e'_2,e'_i\rangle=-\frac13$.
\smallskip

Assume $i\neq j$ with $\{i,j\} \subset \{3,\ldots,n\}$.
If $e'_1-e'_i$ and $e'_1-e'_j$ are both long roots of $\mathfrak{h}$,
they must have an angle $\frac{\pi}{3}$ because
$
\langle e'_1-e'_i,e'_1-e'_j\rangle=\frac53+\langle e'_i,e'_j\rangle>0.
$
So $\langle e'_1-e'_i,e'_1-e'_j\rangle=\frac43$, which implies
$\langle e'_i,e'_j\rangle=-\frac13$.
If $e'_1-e'_i$ is a long root and $e'_1+e'_j$ is short they must have
an angle $\frac{\pi}{4}$ because
$
\langle e'_1-e'_i,e'_1+e'_j\rangle=1-\langle e'_i,e'_j\rangle>0.
$
So $\langle e'_1-e'_i,e'_1+e'_j\rangle=\frac43$, which implies
$\langle e'_i,e'_j\rangle=-\frac13$.
\smallskip

Based on the above observations, we see if there is $e'_i$, $i>2$, such that
$e'_1-e'_i$ is a long root of $\mathfrak{h}$, we can suitably permute $e_j$\,,
$j>2$ to make $i=3$. Then the matrix
$(\langle e'_p,e'_q\rangle)_{1\leqq p,q\leqq 5}$ must be of the form
$$\left(
    \begin{array}{ccccc}
      1 & -\tfrac13 & -\tfrac13 & -\tfrac13 & -\tfrac13 \\
      -\tfrac13 & 1 & -\tfrac13 & -\tfrac13 & -\tfrac13 \\
      -\tfrac13 & -\tfrac13 & 1 & -\tfrac13 & -\tfrac13 \\
      -\tfrac13 & -\tfrac13 & -\tfrac13 & 1 & \tfrac13 \\
      -\tfrac13 & -\tfrac13 & -\tfrac13 & \tfrac13 & 1 \\
    \end{array}
  \right) \mbox{ or }
\left(
    \begin{array}{ccccc}
      1 & -\tfrac13 & -\tfrac13 & -\tfrac13 & -\tfrac13 \\
      -\tfrac13 & 1 & -\tfrac13 & -\tfrac13 & -\tfrac13 \\
      -\tfrac13 & -\tfrac13 & 1 & -\tfrac13 & -\tfrac13 \\
      -\tfrac13 & -\tfrac13 & -\tfrac13 & 1 & -\tfrac13 \\
      -\tfrac13 & -\tfrac13 & -\tfrac13 & -\tfrac13 & 1 \\
    \end{array}
  \right),
$$
which is non-singular in either case.  That contradicts
$\dim(\mathfrak{t}\cap\mathfrak{h})=4$. So $\mathfrak{h}$ has no long root
of the form $e'_1-e'_i$, $i>2$. Furthermore, $\mathfrak{h}$ has no
long root of the form $e'_2-e'_i$, $i>2$; for if
$e'_2-e'_i$ is a long root with $i>2$, then it has an angle
$\tfrac{2\pi}3$ with
the long root $e'_1-e'_2$ of $\mathfrak{h}$, i.e. $e'_1-e'_i=
(e'_1-e'_2)+(e'_2-e'_i)$ is a long root of $\mathfrak{h}$, which contradicts
our previous statement. The number of long roots
with the form $\pm e'_i\pm e'_j$ is at most $22$, consisting of
$\pm(e'_1-e'_2)$, and at most a pair of long roots
from each set $\{\pm e'_i\pm e'_j\}$ for $2<i<j\leqq 7$. The number of long
roots of $\mathfrak{h}$ can not reach $24$, which is a contradiction.

This completes the proof, in the $\mathfrak{f}_4$ case, that
we do not have roots $e'_1\pm e'_2$ of $\mathfrak{h}$ with
$\langle e'_1,e'_2\rangle=\pm\frac13$.
\smallskip

Next, in the $\mathfrak{f}_4$ case, we consider the  situation where
$\langle e'_1,e'_2\rangle=0$ and $e'_1\pm e'_2$ are roots of $\mathfrak{h}$.
Since $\mathfrak{h}$ has no root of length $2$, any short root
$\pm e'_i\pm e'_j$ has the length 1 with $\langle e'_i,e'_j\rangle=\pm\frac12$,
and any long root $\pm e'_i\pm e'_j$ has the length $\sqrt{2}$ with
$\langle e'_i,e'_j\rangle=0$. If $i$, $j$ and $k$ are distinct,
$\langle e'_i,e'_j\rangle=0$, and $\langle e'_i,e'_k\rangle=0$, then for
suitable $c_1=\pm 1$ and $c_2=\pm 1$ the roots $e'_i+c_1 e'_j$ and
$e'_i+c_2 e'_k$ of $\mathfrak{h}$ are long.  Because
$$
\langle e'_i+c_1 e'_j,e'_i+c_2 e'_j\rangle=1\pm\langle e'_i,e'_j \rangle>0,
$$
the combination $c_1 e'_j-c_2 e'_k=(e'_i+c_1 e'_j)-(e'_i+c_2 e'_j)$ is a long
root of $\mathfrak{h}$.  That implies $\langle e'_j,e'_k\rangle=0$.
Now $\{1,\ldots,7\}$ is a disjoint union $\coprod_{a\in\mathcal{A}}S_a$
such that (i) if $i\neq j$ in the same $S_a$ then $e'_i\perp e'_j$
and (ii) if $i\in S_a$ and $j\in S_b$ with $a\neq b$ then
$\langle e_i,e_j\rangle=\pm\frac12$. If $i\neq j$ are in the same
$S_a$\,, and $e_i\pm e_j$ are both long roots of $\mathfrak{h}$, then for
whenever $k\in S_a$, $i\neq k\neq j$,
there is a long root of the form $e_i\pm e_k$\,. It has
angle $\frac{\pi}3$ with both $e_i + e_j$ and $e_i -e_j$\,, so both
$e_k\pm e_j$ are long roots of $\mathfrak{h}$. Similarly, both $e_i\pm e_k$
are long roots of $\mathfrak{h}$. Extending this argument, whenever $k\neq l$
in the same $S_a$, both $e_k\pm e_l$ are long roots of $\mathfrak{h}$.
\smallskip

Each $|S_a| \leqq 4$.  For if  $S_a$ contains five elements then
$\dim\mathfrak{t}\cap\mathfrak{h}>4$ which is impossible.
\smallskip

Suppose $|S_a| = 4$ with $\{1,2\} \subset S_a$\,.
We may permute the $e_i$ so that $S_a=\{1,2,3,4\}$. Then $\{e'_1,\dots,e'_4\}$
is an orthonormal basis of $\mathfrak{t}\cap\mathfrak{h}$. By our previous
observation, $\pm e'_i\pm e'_j$ provide all long roots of $\mathfrak{h}$.
From the standard presentation (\ref{root-system-F-4}) of $\mathfrak{f}_4$,
we can see, for any orthogonal pair of long roots $\alpha'$ and
$\beta'$ of $\mathfrak{h}$, $\frac12(\alpha'\pm\beta')$ are short roots of
$\mathfrak{h}$.  Thus the
$\pm e'_i=\pm\frac12((e'_i+e'_j)+(e'_i-e'_j))$, $1\leqq i\leqq 4$ and $i\neq j$,
are short roots of $\mathfrak{h}$.
But $e'_i=\frac12(\pm e'_1\pm\cdots\pm e'_4)$ for $i=5$, $6$
and $7$, so any short root $\pm e'_i\pm e'_j$ of $\mathfrak{h}$\,, for $1\leqq i\leqq 4<j\leqq 7$,
is a vector of the form $\frac12(\pm e_1\pm\cdots\pm e_4)$. And each set
$\{\pm e'_j\pm e'_k\}$, $4<j<k\leqq 7$, contains at most one pair of short
roots, resulting in 3 pairs in total.  That is not enough for the
presentation just above for all the $\pm e'_i$, $1\leqq i\leqq 4$.
There is at least a short root $e'_i$ of $\mathfrak{h}$, $1\leq i\leq 4$,
which can not be given as any $\pm e'_j\pm e'_k$.
This is a contradiction to our observation for the root system of $\mathfrak{h}$.
So there is no $S_a$ with  $|S_a| = 4$ and
$\{1,2\} \subset S_a$\,.
\smallskip

Assume that one of the sets $S_b$ satisfies $|S_b| = 4$ and it does not contain $1$ and $2$\,.  Then we can permute the $e_i$ so that
$S_b=\{4,5,6,7\}$. From argument above, if $4<i<j\leqq 7$, then either
$e'_i+e'_j$ or $e'_i-e'_j$ is not a root of $\mathfrak{h}$. So
$S_b$ can at most provide 6 pairs of long roots of $\mathfrak{h}$.
The only way that $\mathfrak{h}$ can have 12 pairs of long roots is that
$1$, $2$ and $3$ must belong to the same $S_a$\,, and $\pm e'_i\pm e'_j$
are long roots of $\mathfrak{h}$ for $1\leqq i<j\leqq 3$.  In that case
we look at the short roots.  Each set $\{\pm e'_i\pm e'_j\}$,
$1\leqq i\leqq 4<j\leqq 7$, can only provide one pair of short roots, and in
fact it must provide one pair to make the 12 pairs of short roots of
$\mathfrak{h} = \mathfrak{f}_4$\,. Now, if $\alpha'$ is a root
of $\mathfrak{h}$, then $\mathfrak{h}_{\pm\alpha'}=
\widehat{\mathfrak{g}}_{\pm\alpha'}$ is a root plane of $\mathfrak{g}$,
say $\mathfrak{h}_{\pm\alpha'} = \mathfrak{g}_{\pm\alpha}$\,.  Then
$\alpha'$ is also a root of $\mathfrak{g}$ because
$$
\mathbb{R}\alpha'=[\mathfrak{h}_{\pm\alpha'},
\mathfrak{h}_{\pm\alpha'}]=[\mathfrak{g}_{\pm\alpha},
\mathfrak{g}_{\pm\alpha}]=\mathbb{R}\alpha,
$$
Then the root system of $\mathfrak{h}$ is a subset of the root system of
$\mathfrak{g}$.  That is impossible because all the roots of $\mathfrak{g}$
have the same length.  This completes the argument that none of the
$S_a$ can contain more than 3 elements.
\smallskip

Since it has at most 3 elements, each $S_a$ can contribute no more than
12 long roots. For $\mathfrak{h}$ to have 24 long roots, $\{1,\ldots,7\}
= \coprod_{a\in\mathcal{A}}S_a$ is the union of three subsets, two with
three elements each, one with one element. Suitably permuting the $e_i$ we
can assume $\mathcal{A}=\{a,b,c\}$, $S_a=\{1,2,3\}$, $S_b=\{4,5,6\}$,
and $S_c=\{7\}$. All $\pm e_i\pm e_j$ must be long roots of $\mathfrak{h}$
for $1\leqq i<j\leqq 3$ or $4\leqq i<j\leqq 6$; they give all 24 long
roots of $\mathfrak{h}$.
\smallskip

From above argument, $\pm e'_i$, $1\leqq i<7$, are short roots of $\mathfrak{h}$, because
they can be presented as $\frac12(\pm\alpha'\pm\beta')$ for an orthogonal pair of long roots $\alpha'$ and $\beta'$ of $\mathfrak{h}$. So for $1\leqq i<7$, we can find $j$ and $k$, such that
$e'_i=\pm e'_j\pm e'_k$. It is easy to see, if $i<4$, then $j>3$ and $k>3$,
further more, $j$ and $k$ can not both be chosen from $\{4,5,6\}$, so one of them is from $\{4,5,6\}$, which must be different for different $e'_i$s, and
the other is just $7$. So suitably substitute some $e'_i$s by $-e'_i$s, we
can have $$e'_1+e'_4+e'_7=0,e'_2+e'_5+e'_7=0,\mbox{ and }e'_3+e'_6+e'_7=0,$$
i.e. $\mathfrak{m}$ is linearly spanned by $e_1+e_4+e_7$, $e_2+e_5+e_7$
and $e_3+e_6+e_7$. Direct calculation shows, for $v=e_1$ and $v'=e_7$ in the $\mathrm{Ad}(G)$-orbit of $v$,
$$||\mathrm{pr}_{\mathfrak{m}}(v)||<||\mathrm{pr}_{\mathfrak{m}}(v')||,
$$
which is a contradiction.
\smallskip

To summarize, we have the following proposition.
\begin{proposition}\label{prop-dn}
Suppose that $G$ is a compact connected simple Lie group with
$\mathfrak{g}=\mathrm{Lie}(G)=\mathfrak{d}_n$ where $n>1$, $H$ is a closed
subgroup with $0<\dim H<\dim G$, and $G/H$ is a Riemannian normal
homogeneous space. Assume there is a nonzero vector $v\in\mathfrak{g}$ which
defines a CK vector field on $G/H$. Then $G/H$ is a locally Riemannian
symmetric space which is covered by the sphere
$S^{2n-1}=\mathrm{Spin}(2n)/\mathrm{Spin}(2n-1)=
\mathrm{SO}(2n)/\mathrm{SO}(2n-1)$.
\end{proposition}
\smallskip

Theorem \ref{main} follows by combining
Propositions \ref{easy-no},
\ref{proposition-4-1}, \ref{prop-5-3}, \ref{prop-S7}, \ref{no-cn}
and \ref{prop-dn}.

\end{document}